\documentclass[a4paper]{amsart}
\usepackage{dsfont}
\usepackage{amsmath,amsthm,amssymb,hyperref,mathrsfs}
\usepackage{aliascnt}
\usepackage{lmodern}
\usepackage[T1]{fontenc}
\usepackage[textsize=footnotesize]{todonotes}

\usepackage{xcolor}
\definecolor{dblue}{rgb}{0,0,0.70}
\hypersetup{
	unicode=true,
	colorlinks=true,
	citecolor=dblue,
	linkcolor=dblue,
	anchorcolor=dblue
}

\makeatletter
\expandafter\g@addto@macro\csname th@plain\endcsname{%
		\thm@notefont{\bfseries}
	}%
\expandafter\g@addto@macro\csname th@remark\endcsname{%
		\thm@headfont{\bfseries}
	}%
\makeatother

\newtheorem{theorem}
{Theorem}[section]
\newtheorem*{theorem*}{Theorem}

\newaliascnt{lemma}{theorem}
\newtheorem{lemma}[lemma]{Lemma}
\aliascntresetthe{lemma}
\newtheorem*{lemma*}{Lemma}

\newtheorem{claim}[theorem]{Claim}

\newaliascnt{fact}{theorem}
\newtheorem{fact}[fact]{Fact}
\aliascntresetthe{fact}

\newaliascnt{proposition}{theorem}
\newtheorem{proposition}[proposition]{Proposition}
\aliascntresetthe{proposition}
\newtheorem*{proposition*}{Proposition}

\newaliascnt{corollary}{theorem}
\newtheorem{corollary}[corollary]{Corollary}
\aliascntresetthe{corollary}

\theoremstyle{remark}

\newaliascnt{remark}{theorem}
\newtheorem{remark}[remark]{Remark}
\aliascntresetthe{remark}
\newaliascnt{question}{theorem}
\newtheorem{question}[question]{Question}
\aliascntresetthe{question}
\newaliascnt{conjecture}{theorem}
\newtheorem{conjecture}[conjecture]{Conjecture}
\aliascntresetthe{conjecture}

\newtheorem*{question*}{Question}

\newaliascnt{definition}{theorem}
\newtheorem{definition}[definition]{Definition}
\aliascntresetthe{definition}

\newaliascnt{example}{theorem}

\aliascntresetthe{example}

\renewcommand{\restriction}{\mathbin\upharpoonright}

\newcommand{\axiom}[1]{\mathsf{#1}}
\newcommand{\ZFC}{\axiom{ZFC}}
\newcommand{\AC}{\axiom{AC}}

\newcommand{\AD}{\axiom{AD}}
\newcommand{\DC}{\axiom{DC}}
\newcommand{\ZF}{\axiom{ZF}}
\newcommand{\BPI}{\axiom{BPI}}
\newcommand{\Ord}{\mathrm{Ord}}
\newcommand{\HOD}{\mathrm{HOD}}
\newcommand{\GCH}{\axiom{GCH}}

\newcommand{\IS}{\axiom{IS}}
\newcommand{\HS}{\axiom{HS}}
\newcommand{\KWP}{\axiom{KWP}}
\newcommand{\SVC}{\axiom{SVC}}

\newcommand{\WO}{\axiom{WO}}

\DeclareMathOperator{\cf}{cf}
\DeclareMathOperator{\dom}{dom}

\DeclareMathOperator{\supp}{supp}
\DeclareMathOperator{\rank}{rank}
\DeclareMathOperator{\sym}{sym}

\DeclareMathOperator{\fix}{fix}

\DeclareMathOperator{\id}{id}
\DeclareMathOperator{\aut}{Aut}

\DeclareMathOperator{\Add}{Add}

\DeclareMathOperator{\tcl}{tcl}

\DeclareMathOperator{\SC}{SC}

\newcommand{\forces}{\mathrel{\Vdash}}
\newcommand{\nforces}{\mathrel{\not{\Vdash}}}
\newcommand{\nto}{\mathrel{\nrightarrow}}
\newcommand{\gaut}[1]{{\textstyle\int_{#1}}}

\newcommand{\power}{\mathcal{P}}

\newcommand{\PP}{\mathbb P}
\newcommand{\QQ}{\mathbb Q}
\newcommand{\RR}{\mathbb R}

\newcommand{\cA}{\mathcal A}
\newcommand{\cB}{\mathcal B}

\newcommand{\cF}{\mathcal F}
\newcommand{\cG}{\mathcal G}

\newcommand{\sF}{\mathscr F}
\newcommand{\sG}{\mathscr G}
\newcommand{\sH}{\mathscr H}
\newcommand{\sK}{\mathscr K}

\newcommand{\1}{\mathds 1}

\newcommand{\da}{{\downarrow}}

\newcommand{\tupp}[1]{\left\langle#1\right\rangle}
\newcommand{\tup}[1]{\langle#1\rangle}

\newcommand{\middd}{\mathrel{}\middle|\mathrel{}}

\author{Asaf Karagila}
\thanks{The author was supported by the Royal Society grant no.~NF170989.}
\email{karagila@math.huji.ac.il}
\urladdr{https://karagila.org}
\address{School of Mathematics,
  University of Leeds,
  Leeds, LS2~9JT, UK
}
\date{25 October, 2025}

\title[Approaching a Bristol model]{Approaching a Bristol model$^{(*)}$}
\thanks{(*) A previous version of this text titled ``Guide to the Bristol model: Gazing into the Abyss'' appeared in the RIMS K\^{o}ky\^{u}roku No.2164 (proceedings of the 2019 meeting ``Set theory and Infinity'').}

\begin{document}
\begin{abstract}
The Bristol model is an inner model of $L[c]$, where $c$ is a Cohen real, which is not constructible from a set. The idea was developed in 2011 in a workshop taking place in Bristol, but was only written in detail by the author in \cite{Karagila:Bristol}. This paper is a guide for those who want to get a broader view of the construction. We try to provide more intuition that might serve as a jumping board for those interested in this construction and in odd models of $\ZF$. We also correct a few minor issues in the original paper, as well as prove new results. For example, that the Boolean Prime Ideal theorem fails in the Bristol model, as some sets cannot be linearly ordered, and the ground model is always definable in its Bristol extensions. In addition to this we include a discussion on Kinna--Wagner Principles, which we think may play an important role in understanding the generic multiverse in $\ZF$.
\end{abstract}
\maketitle

\section{Introduction}\label{sect:introduction}
Mathematicians love classifications. We enjoy classifying objects into different categories, and for a good reason. Classifications teach us about abstract properties and help us deepen our understanding of various objects and theories.

Set theorists are generally interested in models of set theory. If $V$ satisfies $\ZFC$, we want to classify models of set theory which lie between $V$ and some generic extension,\footnote{The term ``generic'' will always mean ``set-generic''.} $V[G]$. In the case where ``set theory'' is understood as $\ZFC$, Vop\v{e}nka's theorem tells us exactly what the intermediate models are: they are generic extensions given by subforcings of the forcing which is used to introduce $G$ over $V$.

On the other hand, when we are interested in classifying arbitrary intermediate models of $\ZF$, instead, even if we assume that $V$ satisfied $\ZFC$, the task becomes significantly harder, and dare we say, nigh impossible. For a start, a generic extension of a model of $\ZFC$ cannot be a model of $\ZF+\lnot\AC$. One might be inclined to say that such intermediate extension would still be a \textit{symmetric extension}, which is a type of inner model of a generic extension defined using automorphisms of the forcing. While this is true under some additional conditions on the intermediate model, it turns out that if $M$ is an intermediate model between $V$ and $V[G]$, even if $M$ is a symmetric extension of $V$, it might not be given by any forcing even remotely related to the one for which $G$ was generic.

The reality is that intermediate models of $\ZF$ are far wilder than their $\ZFC$-counterparts. The Bristol model is the first explicit example of such a model. This is a model intermediate to $L[c]$, where $c$ is an $L$-generic Cohen real, which is not $L(x)$ for any set $x$, let alone a symmetric extension of $L$ (by any means, not just the Cohen forcing). While there is a semi-canonical Bristol model, modulo a particular choice of $c$, it is immediate from the construction that, in a very good sense of the word, most models intermediate to $L[c]$ are not even definable. We will clarify this in \autoref{sect:gaps}.

The idea for this model came about in a small 2011 workshop in Bristol on topics related to the HOD Conjecture. In attendance were Andrew Brooke-Taylor, James Cummings, Moti Gitik, Menachem Magidor, Ralf Schindler, Matteo Viale, Philip Welch, and W.~Hugh Woodin, henceforth ``the Bristol group''. The details were not written down in full, and the model remained as a folklore rumour until the author's effort to formalise it. The details of the construction are given in \cite{Karagila:Bristol}, which was part of the author's Ph.D.\ dissertation, written under the supervision of Menachem Magidor. This paper aims to give a bird's eye view of the construction, from three different perspectives (for people coming from different walks of set theory). We will also correct a few minor mistakes in the original paper, and prove a handful of new theorems about the Bristol model, and about models of $\ZF$ in general.

\subsection{Structure of this Paper}\label{sect:structure-this-paper}
The Bristol model is presented in \cite{Karagila:Bristol} as an iteration of symmetric extensions, starting from a Cohen real. The idea is to have, at successor steps, a ``decoding mechanism'' which is a symmetric extension over an intermediate step such that two properties hold: (1) the decoding mechanism has a generic (relative to the intermediate step) in the Cohen extension, and (2) the decoding mechanism only adds subsets of sufficiently high rank.

We will cover the basics of the technical tools in \autoref{sect:tools}. We will define symmetric extensions, and briefly outline the main ideas related to iterating them (or rather, why it is hard to iterate symmetric extensions). We will also discuss the combinatorial ideas needed for the decoding mechanism, both at successors of limits, as well as double successors.

After covering the preliminary tools, we will present the decoding mechanism, and the generic argument needed for the proof to work. In \autoref{sect:cerberus} we explain the three different approaches to constructing the Bristol model. All three are equivalent, but for different people some of these might be seen as ``more natural'' and can help understand the model better. We will not dive into the intimate details, though. The goal of this paper is to serve as a companion, and help provide not only the big picture of the construction, but also serve as a first step towards reading and understanding the construction's details presented in \cite{Karagila:Bristol}.

Having discussed the construction of the model, we will then point towards some minor gaps and typos in the original \cite{Karagila:Bristol}. Then we will discuss Kinna--Wagner Principles which we expect to play a role in the study of choiceless models such as the Bristol model. We will make some new observations, and suggest conjectures for future research. Finally, in \autoref{sect:choice} we will prove that some sets in the Bristol model cannot be linearly ordered, and therefore the Boolean Prime Ideal theorem is false there. We finish the paper with a long list of open questions related to the Bristol model.

\subsection*{Acknowledgements}
The author would like to express his deepest gratitude to Daisuke Ikegami for providing the opportunity to give a long tutorial on the Bristol model during the RIMS Set Theory Workshop in November 2019, ``Set Theory and Infinity'', as well as to the audience, who sat and listened, asked questions, and pointed out difficulties, all of which contributed to this paper. We also want to thank David Asper\'o and Andr\'es E.~Caicedo for providing thorough comments on early versions of this manuscript, as well as the anonymous referee for their helpful suggestions in improving this manuscript.
\newpage

\section{Preliminaries: symmetric extensions, their iterations, and more}\label{sect:tools}
In this paper, the term \textit{forcing} will denote a preordered set with a maximum, denoted by $\1$, unless explicitly mentioned otherwise.\footnote{We will, eventually, do a bit of class forcing.} Of course, we will invariably think about a forcing as a partially ordered set, a separative one, in fact, knowing full well that this will not limit our generality.\footnote{We still insist on the preorder definition, as it does make the definition of an iteration significantly more manageable.} The elements of $\PP$ are called \textit{conditions}, and using them we define \textit{$\PP$-names}. We refer the reader to any of \cite{Halbeisen,Jech,Kunen} for the basic methodology of forcing.

Let $\PP$ be a notion of forcing, we follow the convention that if $p,q\in\PP$, then $q\leq p$ indicates that $q$ is a \textit{stronger} condition, and we will often say that $q$ \textit{extends} $p$. Two conditions are \textit{compatible} if they have a common extension, and they are \textit{incompatible} otherwise.

Given a collection of $\PP$-names, $\{\dot x_i\mid i\in I\}$, that we want to transform into a name, we will denote by $\{\dot x_i\mid i\in I\}^\bullet$ the name $\{\tup{\1,\dot x_i}\mid i\in I\}$, and we say that a name is a \textit{$\bullet$-name} when it has this form. This extends naturally to ordered pairs, sequences, functions, etc. With this notation we can easily define the canonical names for ground model sets: $\check x=\{\check y\mid y\in x\}^\bullet$.

Given two $\PP$-names, $\dot x$ and $\dot y$, we say that $\dot x$ \textit{appears} in $\dot y$ if there is some $p\in\PP$ such that $\tup{p,\dot x}\in\dot y$. We will use a similar terminology stating that $p$ appears in $\dot y$.
\subsection{Symmetric extensions}
As we remarked, a generic extension of a model of $\ZFC$ is again a model of $\ZFC$. Symmetric extensions are intermediate models to generic extensions where the axiom of choice may fail.

Let $\PP$ be a forcing, and let $\pi$ be an automorphism of $\PP$. The action of $\pi$ extends to the $\PP$-names by recursion: \[\pi\dot x=\{\tup{\pi p,\pi\dot y}\mid\tup{p,\dot y}\in\dot x\}.\]

\begin{lemma*}[The Symmetry Lemma]\label{lemma:sym-lemma}
Let $\pi$ be an automorphism of a forcing $\PP$, and let $\dot x$ be a $\PP$-name. For every condition $p$, \[\pushQED{\qed}p\forces\varphi(\dot x)\iff\pi p\forces\varphi(\pi\dot x).\qedhere\popQED\]
\end{lemma*}
Fix a group $\sG\leq\aut(\PP)$. We say that $\sF$ is a \textit{filter of subgroups on $\sG$} if it is a filter on the lattice of subgroups, namely, it is a non-empty collection of subgroups which is closed under finite intersections and supergroups. We will, unless stated otherwise, assume it is a proper filter, i.e.\ the trivial group is not in $\sF$.\footnote{There is no point in using the improper filter when taking a symmetric extension. However, for the sake of generality it should be noted that this can be useful when iterating. We promise to never bring this up in the course of this paper again.} Finally, $\sF$ is \textit{normal} if whenever $\pi\in\sG$ and $H\in\sF$, then $\pi H\pi^{-1}\in\sF$ as well. In most cases we are interested not necessarily in a filter, but in a filter base, and we will ignore the distinction between the two.

Call $\tup{\PP,\sG,\sF}$ a \textit{symmetric system} if $\PP$ is a notion of forcing, $\sG$ is a group of automorphisms of $\PP$, and $\sF$ is a normal filter of subgroups on $\sG$. We shall fix a symmetric system $\tup{\PP,\sG,\sF}$ for the rest of this subsection.

For a $\PP$-name, $\dot x$, let $\sym_\sG(\dot x)$ denote the group $\{\pi\in\sG\mid\pi\dot x=\dot x\}$. If it is the case that $\sym_\sG(\dot x)\in\sF$, then we say that $\dot x$ is \textit{$\sF$-symmetric}. And similarly, we say that $\dot x$ is \textit{hereditarily $\sF$-symmetric} if being $\sF$-symmetric is hereditarily true for all names hereditarily appearing in $\{\dot x\}^\bullet$.

We denote by $\HS_\sF$ the class of hereditarily $\sF$-symmetric names. We denote by $\forces^\HS$ the relativisation of the forcing relation to $\HS$: we restrict the quantifiers and free variables to this class. It is not hard to check that the Symmetry Lemma applies for $\forces^\HS$, provided that we use automorphisms from $\sG$. The following theorem is a condensation of \cite[pp.~253--254]{Jech}.
\begin{theorem*}
Let $G\subseteq\PP$ be a $V$-generic filter, and let $M$ denote the interpreted class $\HS_\sF^G=\{\dot x^G\mid\dot x\in\HS_\sF\}$. Then $M$ is a transitive class model of $\ZF$ such that $V\subseteq M\subseteq V[G]$. Moreover, $M\models\varphi(\dot x^G)$ if and only if there is some $p\in G$ such that $p\forces^\HS\varphi(\dot x)$.\qed
\end{theorem*}
The class $M$ is also called a \textit{symmetric extension} of $V$. It turns out, as shown by Usuba \cite{Usuba:LSP}, that $M$ is a symmetric extension of $V$ if and only if $M=V(x)$ for some $x\in V[G]$. We will discuss this in more detail in \autoref{sect:gaps}.

We will omit $\sG$ and $\sF$ from the notation and terminology when they are clear from context, which is usually what is going to happen.
\subsection{Example}
Let $\PP$ be $\Add(\omega,\omega_1)$. Namely, $p\in\PP$ is a finite partial function from $\omega_1\times\omega\to2$. We let our $\sG$ be the group of permutations of $\omega_1$ acting on $\PP$ in the natural way: $\pi p(\pi\alpha,n)=p(\alpha,n)$. Finally, for $E\subseteq\omega_1$, let $\fix(E)$ denote $\{\pi\in\sG\mid\pi\restriction E=\id\}$, and set $\sF=\{\fix(E)\mid E\in[\omega_1]^{<\omega_1}\}$.\footnote{The keen-eyed reader will notice that this is a filter base, not a filter.}

For every $\alpha<\omega_1$, let $\dot a_\alpha$ be the name of the $\alpha$th Cohen real, $\{\tup{p,\check n}\mid p(\alpha,n)=1\}$, and set $\dot A=\{\dot a_\alpha\mid\alpha<\omega_1\}^\bullet$.
\begin{claim}
For every $\pi\in\sG$ and $\alpha<\omega_1$, $\pi\dot a_\alpha=\dot a_{\pi\alpha}$. Consequently, $\pi\dot A=\dot A$.\qed
\end{claim}
As an immediate corollary, $\dot a_\alpha\in\HS$ for each $\alpha<\omega_1$, as witnessed by $\fix(\{\alpha\})$, and so $\dot A\in\HS$ as well.
\begin{theorem}\label{thm:example-ccc}
$\1\forces^\HS\dot A$ cannot be well-ordered. Consequently, the real numbers cannot be well-ordered, and therefore $\1\forces^\HS\lnot\AC$.
\end{theorem}
\begin{proof}
Let $\dot f\in\HS$, and suppose that $p$ is a condition such that $p\forces^\HS\dot f\colon\dot A\to\check\eta$ for some ordinal $\eta$. Let $E$ be a countable set such that $\fix(E)\subseteq\sym(\dot f)$. We may also assume that $\pi\in\fix(E)$ satisfies $\pi p=p$ by adding a finite set to $E$, and replacing it with $E\cup\{\alpha\mid\exists n\, \tup{\alpha,n}\in\dom p\}$.

Fix $\alpha\notin E$, as $\PP$ is a c.c.c.\ forcing, the set $X=\{\xi<\eta\mid\exists q\leq p: q\forces^\HS\dot f(\dot a_\alpha)=\check\xi\}$ is countable. Note that $p\forces^\HS\dot f(\dot a_\alpha)\in\check X$.

For any $\beta<\omega_1$, let $\pi_\beta$ denote the $2$-cycle $(\alpha\ \beta)$. Therefore, \[\pi_\beta p\forces^\HS\pi_\beta\dot f(\pi_\beta\dot a_\alpha)\in\pi_\beta\check X.\] Easily, $\pi_\beta\in\fix(E)$ if and only if $\beta\notin E$. So for all $\beta\notin E$, $p\forces^\HS\dot f(\dot a_\beta)\in\check X$. In particular, $p$ must force that $\dot f$ has a countable range. As $p$ and $\dot f$ were arbitrary, we in fact have shown that every ordinal-valued function in the symmetric extension defined on $A$ must have a countable range.

To finish the proof we appeal to the c.c.c.\ of the Cohen forcing again, noting that $\omega_1$ is not collapsed, and therefore $\1\forces^\HS|\dot A|\neq\aleph_0$. Therefore there is no injection from $A$ into the ordinals, as wanted.
\end{proof}
The standard arguments are usually presented in a slightly different way. We usually extend $p$ to a condition $q$, which decides the value of $\dot f(\dot a_\alpha)$, and then show that we find $\pi\in\fix(E)$ such that $\pi q$ is compatible with $q$, and $\pi\alpha\neq\alpha$. This then shows that no extension of $p$ can force $\dot f$ to be injective. Chain condition based arguments are not very common in results of this type, and we hope that this paper will help to popularise the idea.\footnote{See \cite{Karagila:Krivine} for more examples of this sort.}
\subsection{Iterations of symmetric extensions}
Iterating symmetric extensions is not an easy task. While some ad-hoc constructions can be found in the literature from the very early 1970s,\footnote{Examples include \cite{Monro:1973},\cite{Sageev:1975},\cite{Sageev:1981}, and to some extent also \cite{Morris:Thesis}.} the only systematic development of such a framework was done by the author in \cite{Karagila:Iterations}, and so far only for finite support iterations. The goal of this section is not to fully develop and explain this technique, but instead provide an intuition as to how the technique works, and what are the difficulties that need to be overcome in the case of the Bristol model.

The intuition behind iterations of symmetric extensions is as naive and simple as it can get. We want to have an iteration of forcing notions, in this case with finite support, and we want to identify a class of names which correspond to the intermediate model that we would get if we were to iterate symmetric extensions one step at a time.

Looking at a two-step iteration, say $\PP\ast\dot\QQ$, we also need to associate $\sG$ and $\sF$ to $\PP$, and $\PP$-names for a group of automorphisms, $\dot\sH\leq\dot\aut(\dot\QQ)$, and a filter of subgroups $\dot\sK$ on $\dot\sH$. Now, a $\PP\ast\dot\QQ$-name is going to be in the iterated symmetric extension if its projection to $\PP$ is in $\HS_\sF$,\footnote{Recall that a $\PP\ast\dot\QQ$-name has a naturally defined ``$\PP$-name of a $\dot\QQ$-name''. Here we want this name to be in $\HS_\sF$.} and it is forced by $\1_\PP$ to be in $\HS_\sK$, that is a hereditarily symmetric name in the second step.

But we can weaken this slightly, and much like we only require that the projected $\PP$-name is forced to be a name in $\HS_\sK$, we can similarly require that the projected name is forced to be equal to some name in $\HS_\sF$. That we, we are allowed to ``mix'' names from $\HS_\sF$ over an antichain.\footnote{In general a class of names $X$ has the mixing property if when $1\forces\dot x\in X$, that is we can find a (pre-)dense set of conditions $p$ and $\dot x_p\in X$ such that $p\forces\dot x=\dot x_p$, then $\dot x\in X$.} So the projected name need only be ``generically equal'' to an $\sF$-symmetric $\PP$-name, which itself needs only be ``generically equal'' to a $\dot\sK$-symmetric $\dot\QQ$-name.

Consider, for example, Cohen's first model, this is a model similar to the example above, replacing $\omega_1$ by $\omega$. Namely, we add an $\omega$-sequence of Cohen reals, permute them, and consider finite stabilisers. Let $\dot a_n$ be the name for the $n$th real, then we can define the name \[\dot a=\{\tup{p,\check m}\mid\exists n(p(n,0)=1\land\forall k<n, p(k,0)=0\land p(n,m)=1)\}.\] In clearer terms, this is the name for the first $\dot a_n$ which contains $0$. We are guaranteed that $\dot a$ will be interpreted as one of the $\dot a_n$. But it is not hard to see that $\dot a$, as defined above, is not stabilised by fixing any finite subset of $\omega$.\footnote{This shows that typically $\HS$ does not have the mixing property. This is not a bad thing, though, as we are often concerned with particulars when working with symmetric extensions, and having to specify witnesses is a good thing.}

Considering names like $\dot a$ seems like an unnecessary complication. But upon a closer examination of the general construction of iterations, we see this idea is somehow necessary. Indeed, the conditions of the iteration are usually defined to be $\tup{p,\dot q}$ such that $p\in\PP$ and $\1_\PP\forces\dot q\in\dot\QQ$ with $\dot q$ coming from a ``sufficiently rich'' set of names. We will call iterations defined this way ``Jech(-style) iterations''. Kunen, in his book, first introduces iterations more naively: the conditions are pairs $\tup{p,\dot q}$ such that $p\in\PP$ and $p\forces\dot q\in\dot\QQ$ with $\dot q$ appearing in $\dot\QQ$. Unlike Jech iterations, this definition does not generalise well, but it is useful for understanding finite support iterations; we will refer to this as ``Kunen(-style) iterations''.

\begin{remark}
  It is worth pointing out at this point that we assume $\ZFC$ holds in our ground model. While the theory of symmetric extensions, as well as that of iterated forcing, can be developed reasonably well in $\ZF$, it is not clear to what extent choice is truly necessary for developing the theory of iterated symmetric extensions. We utilise the mixing property quite significantly in the general theory, and while it is conceivable, and indeed it is our conjecture, that we can remove choice from the assumptions, there is a certain comfort and simplicity in assuming it. Moreover, as we want to start with $L$ as our underlying model, $\ZFC$ is already a given. While $V=L$ is far too strong of an assumption, we will see that at the very least we will want $\GCH$ to hold, which implies choice anyway.
\end{remark}

When working with Jech iterations, we can quickly see that even for the conditions of $\PP\ast\dot\QQ$ to be in the intermediate model, we need to allow this so-called ``mixing property''. And so, if we require it to hold for the conditions, we will need to require it to also hold for the automorphisms of $\dot\QQ$, etc., and so the idea itself is important. We can now proceed towards finding a combinatorial definition that will allow us to directly define the class of names, much like we did with $\HS$ in the case of a single symmetric extension.

\begin{definition}\label{def:respect}
  Let $\PP$ be a notion of forcing and let $\pi$ be an automorphism of $\PP$. We say that $\pi$ \textit{respects} a name $\dot A$ if $\1_\PP\forces\pi\dot A=\dot A$. If $\tup{\PP,\sG,\sF}$ is a symmetric system we say that $\dot A$ is \textit{$\sF$-respected} if there is some $H\in\sF$ such that every $\pi\in H$ respects $\dot A$.
\end{definition}
Easily, every $\sF$-symmetric name is $\sF$-respected, and much like the definition of $\sF$-symmetric before it, this definition lends itself to a hereditary version. We will simply say ``respected'' when $\sF$ is clear from context. If $\dot A$ carries an implicit structure (e.g., a forcing notion) then this structure is also required to be respected.

The idea is that being respected is ``almost'' being symmetric. We will soon weaken this property a bit further to accommodate the ``mixing property'' into the respected names.

Respect is the foremost necessary condition for developing a combinatorial characterisation of iterated symmetric extensions. Given a symmetric system $\tup{\PP,\sG,\sF}$ and a $\PP$-name $\dot\QQ$, in order to define an automorphism of $\PP\ast\dot\QQ$ using some $\pi\in\sG$, the first thing we need to ensure is that $\pi$ respects $\dot\QQ$. Otherwise, $\tup{p,\dot q}\mapsto\tup{\pi p,\pi\dot q}$ is not an automorphism of $\PP\ast\dot\QQ$.

\begin{definition}\label{def:generic-aut}
  Let $\tup{\PP,\sG,\sF}$ be a symmetric system, and let $\tup{\dot\QQ,\dot\sH,\dot\sK}^\bullet$ be a hereditarily respected name for a symmetric system. Suppose that $\pi\in\sG$ and $\dot\sigma$ is a name such that $\1_\PP\forces\dot\sigma\in\dot\sH$, then we denote by $\gaut{\tup{\pi,\dot\sigma}}$ the automorphism defined by $\tup{p,\dot q}\mapsto\tup{\pi p,\pi(\dot\sigma\dot q)}$.
\end{definition}
Here we utilise the mixing property quite significantly, by defining $\dot\sigma\dot q$ to be the name guaranteed to be interpreted as the action of $\dot\sigma$ on the condition $\dot q$. If we were using Kunen iterations, we would have to define $\gaut{\tup{\pi,\dot\sigma}}$ as a partial automorphism which is only defined when $p\forces\dot\sigma\in\sG$. This is not a formal problem, but it makes the actual legwork a lot harder.

We will denote by $\sG\ast\dot\sH$ the group of all such automorphisms. In \cite{Karagila:Iterations} we refer to this group as the \textit{generic semidirect product}. Turning our attention to the filters $\sF$ and $\dot\sK$, we define a \textit{support} to be $\tup{\dot H_0,\dot H_1}$ where both of them are $\PP\ast\dot\QQ$-names such that $\forces\dot H_0\in\check\sF$ and $\dot H_1\in\dot\sK$. Note that we can always extend whatever condition to decide the actual group $\dot H_0$, and likewise we can decide the actual $\PP$-name for $\dot H_1$. However, using this approach allows us to take advantage of the mixing property.

We can now define the notion of respect relative to the supports. Namely, there is a pair $\tup{\dot H_0,\dot H_1}$ such that whenever $\tup{p,\dot q}\forces \tup{\check\pi,\dot\sigma}\in \tup{\dot H_0,\dot H_1}$, which is to say $\tup{p,\dot q}\forces\check\pi\in\dot H_0$ and $\dot\sigma\in\dot H_1$, we have that $\tup{p,\dot q}\forces\gaut{\tup{\pi,\dot\sigma}}\dot A=\dot A$. Whereas in \autoref{def:respect} we required that $\dot H_0$ would actually be a concrete group, here we allow a bit more leeway. If we consider this definition at a single-step forcing, we weakened the requirement that $H\in\sF$ to requiring that $\dot H$ is a name for a group in $\sF$. If we mix names from $\HS_\sF$, then the result will be respected, since we can mix the groups, $\sym(\dot x)$, to obtain a name $\dot H$ which is forced to be in $\sF$.

A \textit{symmetric iteration}, or an iteration of symmetric extensions, is defined by specifying a sequence of names for symmetric systems, $\tup{\dot\QQ_\beta,\dot\sG_\beta,\dot\sF_\beta\mid\beta<\alpha}$, and defining $\PP_\alpha$ as the finite support iteration of the forcings $\dot\QQ_\beta$; $\cG_\alpha$ as the finite support generic semidirect product of the groups $\dot\sG_\beta$, made of sequences $\tup{\dot\pi_\beta\mid\beta<\alpha}$ such that $\1_\beta\forces\dot\pi_\beta\in\dot\sG_\beta$ and for all but finitely many $\beta$, $\1_\beta\forces\dot\pi_\beta=\id^\bullet$; and $\cF_\alpha$ as the collection of all supports, defined as follows.

We define supports in the general case as $\PP_\alpha$-names for sequences $\tup{\dot H_\beta\mid\beta<\alpha}$ such that $\forces_\alpha\dot H_\beta\in\dot\sF_\beta$, and that $\forces_\alpha\{\check\beta\mid\dot H_\beta\neq\dot\sG_\beta\}^\bullet$ is finite. The last part is crucial, as it allows us the flexibility in ``knowing something has a finite definition'' while not committing to its specifics just yet.

Before we continue on, let us expand a bit on these definitions, with the caveat that the generality of the method is not needed for the Bristol model, as will be explained later, and so we encourage the interested reader to study the general framework for iterations as presented in \cite{Karagila:Iterations}. Given a sequence $\vec\pi$ in $\cG_\alpha$, the action on $\PP_\alpha$, denoted by $\gaut{\vec\pi}p$, will be defined by recursion. This is easy if there is exactly one $\beta<\alpha$ for which $\1_\beta\nforces\dot\pi_\beta=\id^\bullet$. In that case $\gaut{\vec\pi\restriction\beta}=\id$ and $\gaut{\dot\pi_\beta}$ is acting on $\QQ_\beta$ as well as on the name of the quotient $\PP_\alpha/\QQ_{\beta+1}$. And so, \[\gaut{\vec\pi}p=p\restriction\beta^\smallfrown\dot\pi_\beta(p(\beta))^\smallfrown\dot\pi_\beta(p\restriction(\beta,\alpha)).\]
Letting $C(\vec\pi)=\{\beta<\alpha\mid\1_\beta\nforces\dot\pi_\beta=\id^\bullet\}=\{\beta_0<\dots<\beta_n\}$, then $\gaut{\vec\pi}$ is simply the composition $\gaut{\dot\pi_{\beta_0}}\dots\gaut{\dot\pi_{\beta_n}}$.\footnote{The action can be down ``upwards'' rather than ``downwards''.} Looking at a support, $\vec H=\tup{\dot H_\beta\mid\beta<\alpha}$, these do not quite define a subgroup of $\cG_\alpha$, however, for each $p\in\PP_\alpha$ we can define a ``local group'' which is generated by $\{\gaut{\vec\pi}\mid\forall\beta<\alpha, p\forces\dot\beta_\beta\in\dot H_\beta\}$.\footnote{In an upcoming work with Jonathan Schilhan we simplify much of this framework, including the concept of supports.} This leads us to the definition of respect in this case. Namely, $\dot x$ is $\cF$-respected if there is a support $\vec H$ such that whenever $p\forces\vec\pi\in\vec H$, then $p\forces\gaut{\vec\pi}\dot x=\dot x$.

Let $\IS_\alpha$ be the class of $\PP_\alpha$-names which are hereditarily respected. This class is now the class of names which will be interpreted in the intermediate model of stage $\alpha$. So to complete our definition of a symmetric iteration we need to require that $\tup{\dot\QQ_\alpha,\dot\sG_\alpha,\dot\sF_\alpha}^\bullet$ is in $\IS_\alpha$, and indeed that it is respected by all automorphisms in $\cG_\alpha$, which guarantees that the action of $\cG_\alpha$ extends to $\PP_{\alpha+1}$. We also have a forcing relation $\forces^\IS_\alpha$ which is defined by relativising the names and quantifiers to $\IS_\alpha$.

Using this definition we can show that when $G$ is $V$-generic, then $\IS_\alpha^G$ is a model of $\ZF$, and if $\alpha=\beta+1$, then this model is a symmetric extension of $\IS_\beta^{G\restriction\beta}$ using the $\beta$th symmetric system. On the other hand, if we want to continue a symmetric iteration, that is of course possible, but we need to make sure that the generic filter used is not only $\IS_\alpha^G$-generic, but rather $V[G]$-generic, and we may need to shrink the groups $\dot\sG_\beta$ in a few places to satisfy the requirement that each symmetric system is respected by the relevant $\cG_\beta$. This may present an issue if we want to continue our iteration by infinitely many steps, as we may need to shrink our groups infinitely many times, which might not be possible, as the filters of groups might not be sufficiently closed.

Here we arrive to our first obstacle when applying this definition to the Bristol model. We required the generic filter is $V$-generic for the entire iteration, whereas the Bristol model is defined inside a single Cohen real, so we need to find a way to modify the definition to allow for ``pointwise genericity'', so that we can construct the Bristol model one step at a time, while finding ``sufficiently generic'' objects inside $L[c]$. In the case of iterated forcing using pointwise genericity is not an issue for successor steps,\footnote{Although it can be an issue for the limit case.} but here we run into the first problem we had defining the whole apparatus using mixing. The use of mixing allows us to use arbitrary antichains and predense sets to define the names in $\IS_\alpha$, and so we end up with names in $\IS_\alpha$ which encode some generic information in them. This means that pointwise generic filters might not be able to correctly interpret such names, since they might not actually meet the relevant antichains.

The ``easy'' way to leave this mess is to require that we relativise the definition pointwise. That is, at each step we only take names which are in the intermediate model. But this adds a layer of complexity when defining the actual forcing $\PP_\alpha$, or the automorphisms in $\cG_\alpha$, or using these in the same manner that we are used to when working with symmetric extensions. Even worse, while for finite support iterations all of these different constructions are equivalent, this is not the case if we want to extend our definition to other types of iterations.

We take a different route instead. Looking at products as a type of degenerate iterations, we may want to mimic this definition here. But a copy-paste approach is bound to result in just a product of symmetric extensions. While this is fine, it is not what we are looking for. We want to force over the symmetric extensions, but with a ``very canonically defined forcing''. The idea is that we want to iterate $\PP\ast\dot\QQ$, and in $V[G]$, where $G\subseteq\PP$, the forcing $\dot\QQ^G$ is isomorphic to a forcing in $V$, and to some extent, this isomorphism does not even depend on $G$. This means that we are really taking a product. But in the symmetric extension given by $\PP$, the forcing $\dot\QQ^G$ is not isomorphic to any partial order in $V$, maybe because it cannot be well-ordered, or maybe due to a similar consideration. Nevertheless, it is distinct from forcings in $V$ as far as $\HS^G$ is concerned.

A symmetric iteration is called \textit{productive} if it behaves, essentially as a product. Namely, each $\dot\QQ_\alpha$, $\dot\sG_\alpha$, and $\dot\sF_\alpha$ are $\bullet$-names, with $\1_\alpha$ deciding all the relevant formulas. Namely, $\1_\alpha$ decides the truth value of when two conditions, automorphisms, or groups, are equal; it decides when $\dot\pi\dot q=\dot q'$, etc. In this case, the symmetric system is ``essentially a copy of a ground model system'',\footnote{Specifically, in the generic extension this will true, but the isomorphism itself might have been non-symmetric, so the ``current'' intermediate model over does not know about it.} and the iteration can be presented as a bona fida product of these names. For this concept to be complete we need to also remove the flexibility in the definition of supports, which is needed for the iterative definition of $\IS_\alpha$. We require, in the productive case, that $\tup{\dot H_\beta\mid\beta<\alpha}$ is an \textit{excellent support}, meaning that $\dot H_\beta$ is a name appearing in $\dot\sF_\beta$, and in particular a $\PP_\beta$-name, and the finite set of non-trivial coordinates is decided in advance. This is in line with the previous demands: everything is decided in advance, this is ``almost a product''.

Now that we have restricted the iteration, we can extend the concept of ``generic''.
\begin{definition}
  Suppose that $\tup{\PP,\cG,\cF}$ is an iteration of symmetric extensions. We say that $D\subseteq\PP$ is \textit{symmetric} if there is an excellent support $\vec H$ such that whenever $p\forces\vec\pi\in\vec H$, then $p\forces \check D=\{\gaut{\vec\pi}q\mid q\in D\}^\bullet$. In other words, $D$ is stable, as a set, under a large group of automorphisms. We say that $D$ is \emph{symmetrically dense} if it is a symmetric dense set.\footnote{And this extends to predense, open, etc.} We say that $G\subseteq\PP$ is \textit{symmetrically $V$-generic} if it is a filter meeting every symmetrically dense set in $V$.
\end{definition}

It turns out that symmetrically generic filters are exactly the filters needed to interpret symmetric names, and that the following theorem holds.
\begin{theorem}
  Let $\PP$ be a productive iteration,\footnote{These include, of course, actual products, as well as single-step symmetric extensions.}and let $p\in\PP$ and $\dot x\in\IS$ be some symmetric name. The following conditions are equivalent:
  \begin{enumerate}
  \item $p\forces^\IS\varphi(\dot x)$.
  \item For every symmetrically $V$-generic filter $G$ such that $p\in G$, $\IS^G\models\varphi(\dot x^G)$.
  \item For every $V$-generic filter $G$ such that $p\in G$, $\IS^G\models\varphi(\dot x^G)$.\qed
  \end{enumerate}
\end{theorem}

It is not hard to check now that at least for the successor steps, the iteration of ``pointwise'' symmetrically generic filters is indeed a symmetrically generic filter.

We finish this overview with a preservation theorem.
\begin{theorem}\label{thm:preservation}
  Suppose that $\tup{\dot\QQ_\alpha,\dot\sG_\alpha,\dot\sF_\alpha\mid\alpha<\delta}$ defines a symmetric iteration, and $G$ is $V$-generic for the iteration. Moreover, assume that for every $\alpha<\delta$, $\1\forces_\alpha\dot\QQ_\alpha$ is weakly homogeneous and $\dot\sG_\alpha$ is rich enough to witness this.\footnote{We will say in this case that $\sG_\alpha$ witnesses the homogeneity of $\QQ_\alpha$.} Let $\eta$ be an ordinal such that there exists $\alpha_0<\delta$ that for any $\alpha\in(\alpha_0,\delta)$ the following equality holds: \[V_\eta^{\IS_\alpha^{G\restriction\alpha}}=V_\eta^{\IS_{\alpha+1}^{G\restriction\alpha+1}}.\]
  Then $V_\eta^{\IS_\delta^G}=V_\eta^{\IS_{\alpha_0+1}^{G\restriction\alpha_0+1}}$. In other words, if no sets of rank ${<}\eta$ were added at successor steps, none were added at limit steps either.\qed
\end{theorem}
This theorem is in stark contrast to the familiar case in the usual context of iterated forcing: iterating, with finite support, forcings which are not c.c.c.\ will collapse cardinals; and iterating non-trivial forcings, even if they are c.c.c., will add Cohen reals at limit steps. But in the case of symmetric iterations, even if the forcings are non-trivial, as long as they are homogeneous and do not add reals, the limit steps will not add reals either.

As a consequence, we can extend our apparatus now to an $\Ord$-length iteration while preserving $\ZF$ in the resulting model. Moreover, the result holds for productive iterations with symmetrically generic filters, as one can state it in the language of forcing, rather than talking about $V_\eta$ of various models. See also \S9.2 of \cite{Karagila:Iterations}.

\subsection{Permutable families and scales}

The key mechanism in the construction of the Bristol model is ``decoding a long sequence from a short sequence''. This can mean a sequence of length $\omega_1$ from a Cohen real, or a sequence of length $\omega_{43}$ from one of length $\omega_{42}$. We use almost disjoint families in successor steps to repeatedly decode these sequences, and we use a particular type of a PCF-scale to succeed at this task when we are at limit steps. This will be as good a place as any to remind the reader that we are working in $\ZFC$, especially when thinking about these combinatorial objects that are used here.

\begin{definition}
  Let $\kappa$ be a regular cardinal, and fix a family $\cA=\{A_\alpha\mid\alpha<\kappa^+\}$ of unbounded subsets of $\kappa$ such that for $\alpha<\beta$, $\sup(A_\alpha\cap A_\beta)<\kappa$. For permutations $\pi\colon\kappa\to\kappa$ and $\Pi\colon\kappa^+\to\kappa^+$ we say that that \textit{$\pi$ implements $\Pi$} if $\pi``A_\alpha=^*A_\Pi(\alpha)$ for all $\alpha<\kappa^+$.
\end{definition}

Here we use $=^*$ to mean equality up to a bounded subset of $\kappa$. Which $\kappa$, of course, will be clear from the context, so we will spare the notation $=^*_\kappa$, or worse, from the reader.

We are looking for an abstract property of an almost disjoint family which will ensure that it implements any bounded permutation of $\kappa^+$, that is any permutation of $\kappa^+$ which is the identity on a tail can be implemented.

\begin{definition}
  Let $\kappa$ be a regular cardinal, $\{A_\alpha\mid\alpha<\kappa^+\}\subseteq[\kappa]^\kappa$ is called a \textit{permutable family} if it is almost disjoint, that is for $\alpha\neq\beta$, $\sup(A_\alpha\cap A_\beta)<\kappa$, and for every $I\in[\kappa^+]^{<\kappa^+}$ there is a pairwise disjoint family $\{B_\xi\mid\xi\in I\}$ such that $B_\xi=^*A_\xi$ and  \[\alpha\in I\iff A_\alpha\cap\bigcup_{\xi\in I} B_\xi\text{ is unbounded in }\kappa.\]
\end{definition}

We call $\{B_\xi\mid\xi\in I\}$ as in this definition a \textit{disjoint approximation}, and in the case where $B_\xi\subseteq A_\xi$, we say it is a \textit{disjoint refinement}. Note that both of these always exist, as long as $|I|\leq\kappa$, by a standard transfinite recursion argument using the regularity of $\kappa$.

\begin{proposition}
  If $\cA$ is a permutable family of subsets of a regular cardinal $\kappa$, then it implements every bounded permutation of $\kappa^+$.
\end{proposition}
\begin{proof}
  Let $\Pi$ be a bounded permutation of $\kappa^+$, fix $\eta<\kappa^+$ such that $\Pi$ does not move any ordinals above $\eta$. Next, set $I=\eta$ and let $\{B_\xi\mid\xi<\eta\}$ be a disjoint refinement. Now let $\pi$ be the function which is the order isomorphism from $B_\alpha$ to $B_\beta$ when $\Pi(\alpha)=\beta$, and the identity elsewhere. Easily, $\pi$ implements $\Pi$.
\end{proof}

Having fixed a permutable family, if $\pi\colon\kappa\to\kappa$ implements $\Pi$, we will denote this by $\iota(\pi)=\Pi$.

\begin{proposition}
Let $\kappa$ be a regular cardinal, then a permutable family exists.
\end{proposition}
\begin{proof}
  Let $\tup{T_\alpha\mid\alpha<\kappa^+}\subseteq[\kappa]^\kappa$ be a $\subsetneq^*$-increasing family of subsets of $\kappa$. Define $A_\alpha$ as $T_{\alpha+1}\setminus T_\alpha$, then $\{A_\alpha\mid\alpha<\kappa^+\}$ is a permutable family. To see that, let $I\in[\kappa^+]^{<\kappa^+}$ and let $\eta=\sup I+1$, then every $A_\alpha$ for $\alpha\in I$ is almost contained in $T_\eta$. We let $B_\alpha$ for $\alpha<\eta$ be a disjoint refinement (of a potentially larger set) such that $B_\alpha\subseteq T_\eta$, in particular $\bigcup_{\xi\in I}B_\xi\subseteq T_\eta$. If $\alpha\notin I$, then either $\alpha<\eta$, in which case $A_\alpha=^*B_\alpha$, and $B_\alpha\cap\bigcup_{\xi\in I}B_\xi=\varnothing$, or else $\alpha\geq\eta$, in which case $A_\alpha\cap T_\eta$ is bounded in $\kappa$.
\end{proof}

\begin{remark}
  It should be pointed out that one can construct an increasing family of subsets from a permutable family. Recursively, set $T_0=A_0$, $T_{\alpha+1}=T_\alpha\cup A_\alpha$, and for limit steps recursively construct $T_\alpha$ as the union of a disjoint refinement of $\{A_\beta\mid\beta<\alpha\}$, which exists by definition of a permutable family. As long as $\alpha<\kappa^+$ we can ensure that these disjoint refinements are also increasing in inclusion, which guarantees that $T_\alpha$ contains previous limit steps.
\end{remark}

\begin{definition}
  Given a permutable family on a regular cardinal $\kappa$, the \textit{derived group} is the group $\sG$ of all permutations of $\kappa$ which implement a bounded permutation of $\kappa^+$. The \textit{derived filter} is the normal filter of subgroups on $\sG$ generated by $\fix(\cB)$, for a disjoint approximation $\cB$, where $\fix(\cB)=\{\pi\in\sG\mid\pi\restriction\bigcup\cB=\id\}$.
\end{definition}

While these definitions are given for regular cardinals, we will only use them in the basis case and successor case. For the limit case, where $\kappa$ is a limit cardinal, we need to use a slightly different machinery, as the goal is to coalesce the information from previous steps and use it as a kind of ``short sequence''. In some way, inaccessible cardinals are the ``simpler case'' compared to singular cardinals. Nevertheless, there is no need of separating the two.

\begin{definition}
  Let $\lambda$ be a limit cardinal, and let $\SC(\lambda)$ denote $\{\mu^+\mid\mu<\lambda\}$. Let $\{f_\alpha\mid\alpha<\lambda^+\}$ be a scale in $\prod\SC(\lambda)$.\footnote{This is an increasing sequence of functions in the product which is bounding in the eventual domination order. We actually only need it to be an increasing sequence, it is irrelevant that it is also bounding.} Given a sequence of permutations $\vec\pi=\tup{\pi_\theta\mid\theta\in\SC(\lambda)}$ such that $\pi_\theta$ is a permutation of $\theta$, we say that $\vec\pi$ \textit{implements} a function $\Pi\colon\lambda^+\to\lambda^+$ if for every $\alpha<\lambda^+$ and for every large enough $\theta$, \[ f_{\Pi(\alpha)}(\theta) = \pi_\theta f_\alpha(\theta).\]

  We call a scale \textit{permutable} if it implements every bounded permutation of $\lambda^+$.
\end{definition}

As with the case of permutable families, we denote by $\iota(\vec\pi)$ the permutation $\Pi$ that is implemented by $\vec\pi$.

The fact that inaccessible cardinals are the ``simpler case'' can be trivially seen as a consequence of the following proposition, while remembering that working in $V=L$ we always have the wanted cardinal arithmetic.
\begin{proposition}
Suppose that $\lambda$ is an inaccessible cardinal and $2^\lambda=\lambda^+$, then there is a permutable scale on $\lambda$.\footnote{The assumption on $2^\lambda$ can be completely removed by simply limiting ourselves to increasing sequences instead of scales. Nevertheless, as we are working under $\GCH$ anyway, this is just simpler.}\qed
\end{proposition}
This can be shown by a simple transfinite recursion, in a very similar fashion to the permutable family case.

\begin{proposition}
Suppose that $\lambda$ is a singular cardinal and $\square^*_\lambda$, then there is a permutable scale on $\lambda$.\qed
\end{proposition}
The proof of this proposition can be found as Theorem~3.27 in \cite{Karagila:Bristol}. The idea of the proof goes back to the Bristol group, and utilises the work of Cummings, Foreman, and Magidor in \cite{CummingsForemanMagidor:Scales} where it is shown that $\square^*_\lambda$ implies the existence of ``better scales''. The aforementioned Theorem~3.27 show that a better scale is in fact permutable.

The key point in proving a better scale is a permutable scale is that given any $I\in[\lambda^+]^{<\lambda^+}$, we can find a function $d\colon I\to\SC(\lambda)$ such that $\{f_\alpha``[d(\alpha),\lambda)\mid\alpha\in I\}$ is a family of pairwise disjoint sets. The proof of Theorem~3.27 also shows that even if we are only allowing $\pi_\theta$ to be a bounded permutation of $\theta$, this is still enough to ensure that we can implement every bounded permutation of $\lambda^+$ using a permutable scale. We say that a sequence of permutation groups of each $\theta\in\SC(\lambda)$ is \textit{rich enough} if we can require $\pi_\theta$ to be in the relevant group when we find an implementing sequence.

\begin{definition}
  Let $\lambda$ be a limit cardinal, and for every $\theta\in\SC(\lambda)$, let $\sG_\theta$ be a rich enough group of permutations of $\theta$. The \textit{derived group} is the subgroup $\sG$ of the full support product $\prod_{\theta\in\SC(\lambda)}\sG_\theta$ consisting of all sequences $\vec\pi=\tup{\pi_\theta\mid\theta\in\SC(\lambda)}$ which implement a bounded permutation of $\lambda^+$.

  For $\eta<\lambda^+$ and $f\in\prod\SC(\lambda)$ we let $K_{\eta,f}$ be the group \[\{\vec\pi\in\sG\mid\iota(\vec\pi)\restriction\eta=\id\text{ and for all }\theta\in\SC(\lambda), \pi_\theta\restriction f(\theta)=\id\}.\] The \textit{derived filter} is the filter generated by $\{K_{\eta,f}\mid\eta<\lambda^+,f\in\prod\SC(\lambda)\}$.
\end{definition}
We can weaken the definition of $K_{\eta,f}$ and replace $\eta$ by a bounded subset of $\lambda^+$, i.e., $I\in[\lambda^+]^{<\lambda^+}$, and replace $f$ by a sequence of bounded sets of each $\theta$. But as the definition is complicated enough as it is, it is easier to just use $\eta$ and $f$ as upper bounds.

\section{Decoding long sequences}\label{sect:decoding}
Assume $V=L$ throughout this section. The Bristol model is constructed as a symmetric iteration, indeed a productive iteration. We will outline the construction of the different intermediate steps in this iteration, and the arguments needed for utilising the iterations apparatus.
\begin{definition}
A \textit{Bristol sequence} is a sequence indexed by the ordinals such that for $\alpha=0$ or $\alpha$ successor we have $\cA_\alpha=\{A^\alpha_\xi\mid\xi<\omega_{\alpha+1}\}$ which is a permutable family on $\omega_\alpha$, and if $\alpha$ is a limit ordinal then we have $F_\alpha=\{f^\alpha_\xi\mid\xi<\omega_{\alpha+1}\}$ which is a permutable scale in the product $\prod\SC(\omega_\alpha)$.
\end{definition}
Fix a Bristol sequence. By assuming $V=L$ we not only have a Bristol sequence, indeed we have a canonical one, where we choose the $<_L$-minimum permutable family or scale at each point. Our goal is to use these permutable objects and replace at each stage, $\alpha$, a sequence of length $\omega_\alpha$ by one of length $\omega_{\alpha+1}$. In particular the first step is to replace the Cohen real, which is a sequence of length $\omega$, by a sequence of length $\omega_1$. But we want to be able to guarantee that the original sequence is \textbf{not} going to be definable from our longer sequence.

In this sense, we want to decode from a Cohen real a sequence of length $\omega_1$, which captures ``some crucial bits'' of the Cohen real, but not really all of it. Then we want to decode from this $\omega_1$ sequence a new sequence of length $\omega_2$, forget the one of length $\omega_1$, and proceed.

\subsection{Example: first steps}
Let $\PP$ be the Cohen forcing, and in this case we mean $p\in\PP$ is a function from a finite subset of $\omega$ into $2$. Let us omit the index from $\cA_0$, as we are only concerned with the first step at the moment, so $A_\alpha$ denotes the $\alpha$th set in the first permutable family. We let $\sG$ and $\sF$ denote the derived group and filter from $\cA$. The action of $\sG$ on $\PP$ is the natural one:\[\pi p(\pi n)=p(n).\]

We denote by $\dot c$ the canonical name for the Cohen real. For $A\subseteq\omega$ let $\PP\restriction A$ denote the subforcing $\{p\in\PP\mid\dom p\subseteq A\}$, and let $\dot c_A$ denote the name \[\{\tup{p,\check n}\mid p(n)=1\land\dom p\subseteq A\}.\] Of course, $\dot c_A$ is the canonical name of the generic real added by $\PP\restriction A$. We have now that $\pi\dot c_A=\dot c_{\pi``A}$. In general we say that a name $\dot x$ for a set of ordinals is \textit{decent} if every name appearing in it is of the form $\check \xi$ for some $\xi\in\Ord$. We say that a name for a set of ordinals is an \textit{$A$-name} if it is a $\PP\restriction A$-name.

\begin{proposition}\label{prop:reals-support}
Suppose that $\dot x\in\HS$ and $\1\forces\dot x\subseteq\omega$, then there is some disjoint approximation $\cB$ and a decent $\bigcup\cB$-name $\dot x_*$ such that $\1\forces\dot x=\dot x_*$.
\end{proposition}
\begin{proof}
 Let $\cB$ be a disjoint approximation such that $\fix(\cB)\subseteq\sym(\dot x)$ and define $\dot x_*$ as \[\dot x_* = \left\{\tup{p,\check n}\middd p\forces\check n\in\dot x\land\dom p\subseteq\bigcup\cB\right\}.\]
 It is clear that $\1\forces\dot x_*\subseteq\dot x$, to show equality it is enough to prove that if $p\forces\check n\in\dot x$, then $p\restriction\bigcup\cB\forces\check n\in\dot x$. Let $q\leq p\restriction\bigcup\cB$.

 By the very definition of a disjoint approximation $\bigcup\cB$ is co-infinite, so we may find a finite set $E$ such that $E\cap(\bigcup\cB\cup\dom p)=\varnothing$ and $|E|=|\dom q\setminus\bigcup\cB|$. Then let $\pi$ be a permutation which extends a bijection between the two sets $\dom q\setminus\bigcup\cB$ and $E$, and is the identity elsewhere. Being a finitary permutation it implements the identity function, so indeed $\pi\in\sG$, and by its very definition $\pi\in\fix(\cB)$, so $\pi\dot x=\dot x$. So if $q$ had forced $\check n\notin\dot x$, we would have $\pi q\forces\pi\check n\notin\pi\dot x$, which is the same as $\pi q\forces\check n\notin\dot x$. Alas, $\pi q$ and $p$ are clearly compatible, and so this is impossible.
\end{proof}

We let $G$ be a $V$-generic filter\footnote{Or $L$-generic filter, to be explicit.} and let $M$ denote the symmetric extension $\HS^G$, as is standard, we will ``omit the dot'' to indicate the interpretation of a name, so $c$ is going to be $\dot c^G$, etc. The following is a very easy corollary from the above proposition.
\begin{corollary}
$c\notin M$.\qed
\end{corollary}

This is where we start seeing the importance of the permutable family, as opposed to any almost disjoint family. If $\{X_\alpha\mid\alpha<\omega_1\}$ is an almost disjoint family, then $\{c\cap X_\alpha\mid\alpha<\omega_1\}$ is a family of mutually generic Cohen reals, any finitely many of them are mutually generic over any other finite subfamily. But in our case, where the almost disjoint family is in fact a permutable one we get \textbf{countable} mutual genericity. Any countable subset of $\{c\cap A_\alpha\mid\alpha<\omega_1\}$ is simultaneously mutually generic over any other countable subset (provided they are pairwise disjoint).

We can actually prove more. Nothing in the proof of \autoref{prop:reals-support} will change if we assume that $\dot x$ is a name of an arbitrary set of ordinals. This shows that every set of ordinals lies in an intermediate model given by $c\cap\bigcup\cB$ for some disjoint approximation. Another way to see this fact is to note that $L[c]$ is, after all, only a Cohen extension by a single real. So every set of ordinals is constructible from a single real.

\begin{corollary}
$M\models\lnot\AC$.
\end{corollary}
\begin{proof}
  Suppose that $M\models\AC$, then there is a set of ordinals $A$ which codes $\RR^M$ (e.g.\ by stacking the real numbers one after another, or by the usual coding of a set into a set of ordinals). Therefore there is $r\in M$ such that $\RR^M=\RR^{L[r]}$. By \autoref{prop:reals-support} we see this is impossible, indeed, if $\cB$ is a disjoint approximation for which $r$ has a $\bigcup\cB$-name, and $A_\alpha$ is such that $A_\alpha\cap\bigcup\cB$ is finite, then $c\cap A_\alpha\notin L[r]$.
\end{proof}
Of course, we can prove directly that $\RR^M$ cannot be well-ordered in $M$, and this argument can be found in the proof of Theorem~2.7 in \cite{Karagila:Bristol}.

We can see \autoref{prop:reals-support} as somehow indicating not only that the reals of $M$ are generated by countable parts of $\cA$, but in fact if $\tup{T_\alpha\mid\alpha<\omega_1}$ is a tower generated by $\cA$, then we actually have that $\RR^M$ is the increasing union of $\RR^{L[c\cap T_\alpha]}$. This was the original approach of the Bristol group.

At this point, one might expect that the decoded sequence is $\tup{c\cap A_\alpha\mid\alpha<\omega_1}$ or somehow $\tup{c\cap T_\alpha\mid\alpha<\omega_1}$. But of course, this is not the case. For starters, we want to somehow ``fuzzy out'' some of the information as to guarantee that $c$ is not constructible from the sequence. So instead of $c\cap A_\alpha$, we will look at $\RR^{L[c\cap A_\alpha]}$. But more importantly, the decoded sequence is not even in $M$. Indeed, if we want this sequence to play the role of the Cohen real in the next step, that means that it needs to be forced into $M$ instead.

Let us begin by understanding $\RR^{L[c\cap A_\alpha]}$. In what way does this set ``fuzzy out'' some information? Well, for one, $c\cap A_\alpha$ is not the obvious real from which we construct this model. Indeed, any finite modification would work, and many more. In fact, any $\pi\in\sG$ for which $\iota(\pi)(\alpha)=\alpha$ will satisfy that $\pi\dot c_{A_\alpha}$ is a name generating the same set of reals, as it is $c\cap A_\alpha$ up to a permutation of $A_\alpha$ and a finite set.

We say that a name $\dot x$ is an \textit{almost $A$-name} if there exists $B$ such that $A=^*B$ and $\dot x$ is a $B$-name. We now define \[\dot R_\alpha=\{\dot x\mid\dot x\text{ is a decent almost }A_\alpha\text{-name}\}^\bullet.\]
\begin{proposition}\label{prop:re-order}
For every $\pi\in\sG$, $\pi\dot R_\alpha=\dot R_{\iota(\pi)(\alpha)}$. In particular, $\dot R_\alpha\in\HS$ for all $\alpha$, and $\{\dot R_\alpha\mid\alpha<\omega_1\}^\bullet\in\HS$ as well.
\end{proposition}
\begin{proof}
Observe that $\pi``\PP\restriction A=\PP\restriction\pi``A$. Therefore if $\iota(\pi)(\alpha)=\beta$ we have that $\pi``A_\alpha=^*A_\beta$, and so an almost $A_\alpha$-name is moved to an almost $A_\beta$-name, so $\pi\dot R_\alpha=\dot R_{\iota(\pi)(\alpha)}$ as wanted. We now have that $\{A_\alpha\}$ is a disjoint approximation for which $\fix(\{A_\alpha\})\subseteq\sym(\dot R_\alpha)$, and thus witnessing that $\dot R_\alpha\in\HS$, and indeed $\sG=\sym(\{\dot R_\alpha\mid\alpha<\omega_1\}^\bullet)$ as wanted.
\end{proof}
Similar arguments as we have seen so far also prove the following statement.
\begin{proposition}\label{prop:sequence-and-inits}
$\tup{\dot R_\alpha\mid\alpha<\omega_1}^\bullet\notin\HS$, but for every countable $I\in[\omega_1]^{<\omega_1}\cap L$, $\tup{\dot R_\alpha\mid\alpha\in I}^\bullet\in\HS$.\qed
\end{proposition}
And here we arrive to the key point. Let $\varrho$ denote the sequence $\tup{R_\alpha\mid\alpha<\omega_1}$ and let $\dot\varrho$ denote its name. We will also write $\varrho_I$ and $\dot\varrho_I$ for the restriction of the sequence to $I$, similar to $\dot c_A$.\footnote{We will never use $\varrho_\eta$ to denote $R_\eta$, so there will not be any confusion in those cases. That is, $\varrho_\eta$ will only ever be used to denote the restriction of the sequence.} Indeed, $\varrho$ is going to be the decoded sequence, and it is of course going to be $M$-generic, but we need to find a suitable partial order. \autoref{prop:sequence-and-inits} provides us with a good clue: every initial segment of this sequence is in fact in $M$, so we can ``safely'' approximate this sequence.

It will be somewhat more convenient to use subsets of $\omega_1$, as we did with the Cohen forcing, rather than proper initial segments. This makes it easier to talk about $A$-names and almost $A$-names. And again for convenience (and so we can claim productivity, of course), we are also going to limit ourselves to subsets of $\omega_1$ which are already in $L$.

\begin{proposition}
Let $\dot\QQ$ denote $\{\pi\dot\varrho_I\mid\pi\in\sG, I\in[\omega_1]^{<\omega_1}\}^\bullet$, ordered by reverse inclusion. Then $\dot\QQ\in\HS$, and indeed $\forces^\HS\dot\varrho$ is $\HS$-generic.
\end{proposition}
In other words, $\varrho$ is $M$-generic for $\QQ$. The idea behind $\QQ$ is that we want the smallest ``reasonable'' set which contains our generic filter (i.e.\ partial approximations of $\varrho$), and the easiest way to do that is to simply apply all permutations and obtain a set. But the true intuition behind $\QQ$, and really behind the whole decoding apparatus, comes from understanding $\pi\varrho_I$.

The model $M$ knows of the set $R=\{R_\alpha\mid\alpha<\omega_1\}$, it just does not know a well-ordering of this set. And we are trying to remedy that. As we know already every $\pi\in\sG$ implements a permutation of $\omega_1$, which induces a permutation of $R$ moving countably many points. So $\pi$ shuffles $R$ and thus modifies the range of $\varrho_I$. But $R$ is an extremely impoverished set as far as $M$ is concerned. This is not particularly important for the construction, and can be skipped entirely, but it is an interesting fact.

\begin{proposition}\label{prop:R-amorphous}
$M\models R$ is a strongly $\aleph_1$-amorphous set. That is, $R$ cannot be written as a union of two uncountable sets, and every uncountable partition of $R$ has at most countably many non-singleton cells.
\end{proposition}
\begin{proof}
  Suppose that $\dot X,\dot Y\in\HS$ and $p\forces^\HS``\dot X,\dot Y\subseteq\dot R$ and are uncountable''. Let $\cB$ be a disjoint approximation such that $\fix(\cB)\subseteq\sym(\dot X)\cap\sym(\dot Y)$ and such that $\dom p\subseteq\bigcup\cB$. By uncountability, we can extend $p$ to some $q$ for which there are $\alpha,\beta$ such that:
  \begin{enumerate}
  \item $q\forces^\HS \dot R_\alpha\in\dot X$ and $\dot R_\beta\in\dot Y$.
  \item $A_\alpha\cap\bigcup\cB$ and $A_\beta\cap\bigcup\cB$ are finite.
  \end{enumerate}
  By enlarging one of the sets in $\cB$, if necessary, we can also assume that $\dom q\subseteq\bigcup\cB$ as well. We can now find a permutation $\pi\in\fix(\cB)$ such that $\iota(\pi)$ is the $2$-cycle switching $\alpha$ and $\beta$.

  Applying $\pi$ to the first property of $q$ we have that $\pi q\forces^\HS\pi\dot R_\alpha\in\pi\dot X,\pi\dot R_\beta\in\pi\dot Y$. But since $\pi\in\fix(\cB)$ we have that $\pi q=q$ and $\pi\dot X=\dot X$ and $\pi\dot Y=\dot Y$. This means that $q\forces\dot R_\beta\in\dot X$ and $\dot R_\alpha\in\dot Y$. In particular $q$ forces that $\dot X$ and $\dot Y$ are not disjoint.

  Next we want to prove that every partition of $R$ is almost entirely singletons. We will only sketch the idea behind the argument. If $\dot S\in\HS$ and $p\forces^\HS``\dot S$ is a partition of $\dot R$ into uncountably many cells'', let $\cB$ be an approximation such that $\fix(\cB)\subseteq\sym(\dot S)$. Pick $\alpha,\beta$ as above, so that we may switch between them without interfering with $\cB$, and we can implement the $2$-cycle $(\alpha\ \beta)$ without changing any given condition. This means that any point that was in the same cell as $\alpha$ must have moved to the cell containing $\beta$. But we only moved two points, so $\alpha$ and $\beta$ must have been isolated as singletons. And therefore the only non-singleton cells come from a partition of $\cB$ itself.
\end{proof}
\begin{remark}
  The above implies that every permutation of $R$ in $M$ only moves countably many points, of course, as the orbits define a partition. Some of these are new, as they can be encoded by some generic real, but this is irrelevant. We can prove that every permutation of $R$ in $M$ only moves countably many points with a direct argument in the style of \autoref{thm:example-ccc}: given a name for a permutation $\dot f$ and $\cB$ a disjoint approximation such that $\fix(\cB)\subseteq\sym(\dot f)$, the c.c.c.\ condition ensures that there is $\delta$ such that for any $\alpha$ for which $A_\alpha\cap\bigcup\cB$ is infinite, $\dot R_\delta$ is not in the same orbit as $\dot R_\alpha$. But now we can utilise the same strategy as we did before and move $\alpha$ to some other ordinal with a similar property, and therefore showing that either $f$ is the identity on a cocountable set, or it is constant there.
\end{remark}

Getting back to the matter at hand, we want to prove that $\varrho$ is $M$-generic. Namely, if $D\subseteq\QQ$ is a dense subset and $D\in M$, we want to prove that there is some $\alpha<\omega_1$ such that $\varrho_\alpha\in D$. In \cite{Karagila:Bristol} we prove this by proving a technical lemma about names of dense open sets, Lemma~2.12.\footnote{There is a minor mistake in the statement of the original lemma, see \autoref{correction:2.12} for details and corrections.} In the paper we use this lemma also as a means for proving that $\QQ$ is $\sigma$-distributive, and therefore does not add any new reals to the model (and as a corollary, it does not force $\AC$ back into the universe somehow), which also finishes the proof that $\varrho$ is indeed the sequence we are looking for.

We will prove the genericity of $\varrho$ using a simplified version of Lemma~2.12, and provide a separate argument for the distributivity.

\begin{proposition}\label{prop:genericity-1}
$\varrho$ is $M$-generic. In other words, if $\dot D\in\HS$ and $p\forces``\dot D$ is a dense open subset of $\dot\QQ$'', then there is some $\eta$ such that $p\forces\dot\varrho_\eta\in\dot D$.
\end{proposition}
\begin{proof}
  Let $\dot D$ be a name as above, and let $\cB$ be a disjoint approximation such that $\fix(\cB)\subseteq\sym(\dot D)$. Let $p\forces``\dot D$ is a dense open subset of $\dot\QQ$''. Let $\alpha$ be large enough such that if $A_\xi\cap\bigcup\cB$ is infinite, then $\xi<\alpha$. Our strategy is to find ``enough'' extensions of $\dot\varrho_\alpha$ so that one of them will be both $\dot\varrho_\eta$, and in $\dot D$.

  First we prove the following claim: suppose that $q\leq p$ and $q\forces\pi\dot\varrho_A\in\dot D$, where $\alpha\subseteq A$ and $\iota(\pi)\restriction\alpha=\id$, then $q\forces\dot\varrho_A\in\dot D$. In other words, if $\pi\dot\varrho_A$ is an extension of $\dot\varrho_\alpha$ which lies in $\dot D$, then $\dot\varrho_A$ is in $\dot D$ as well. Of course, this is true because there is some $\tau\in\fix(\cB)$ such that $\iota(\tau)=\iota(\pi)^{-1}$ and $\tau q=q$.

As $\pi$ implemented the identity up to $\alpha$, $\iota(\pi)^{-1}$ will also be the identity up to $\alpha$, and thus we can implement it using an automorphism in $\fix(\cB)$ which fixes $q$, since $\dom q$ is a finite set.

This completes the proof of the claim since \[\tau q=q\forces\tau\pi\dot\varrho_A=\dot\varrho_{\iota(\tau)\circ\iota(\pi)``A}=\dot\varrho_A\in\tau\dot D=\dot D.\]
In turn, this is enough to prove the genericity of $\varrho$: find a maximal antichain below $p$ of conditions which decide some $\pi\dot\varrho_A$ as above, and by openness we can assume each such $A$ is in fact an ordinal. Let $\eta$ be the supremum of this countable set of ordinals, and we have that $p\forces\dot\varrho_\eta\in\dot D$.
\end{proof}

The final claim is that $\QQ$ does not add new reals to $M$. This, as we remarked, ensures that $M[\varrho]\neq L[c]$. It has an added effect that $\QQ$ is not adding any new sets of ordinals, or any countable sequences of ground model objects. In \cite{Karagila:Bristol} the proof utilised a stronger version of the above proof which lets us intersect a countable sequence of dense open sets. Here we take a slightly different approach.\footnote{The approach we take here is mentioned in a remark at the end of \S2 of \cite{Karagila:Bristol}.}

\begin{proposition}\label{prop:dist-1}
$M\models``\QQ$ is $\sigma$-distributive''. Namely, given $\tup{D_n\mid n\in\omega}\in M$ such that each $D_n$ is a dense open subset of $\QQ$, then $\bigcap_{n\in\omega}D_n$ is dense.
\end{proposition}
\begin{proof}
In $L[c]$, $\QQ$ is naturally isomorphic to $\Add(\omega_1,1)^L$. In $L$ this forcing is $\sigma$-closed, and so in $L[c]$ it is still distributive. If $\tup{D_n\mid n\in\omega}$ is a sequence of dense open sets, then its intersection is still dense in $L[c]$, and therefore in $M$.
\end{proof}

\subsection{Outline of successor steps}
The first two steps are fairly indicative of the standard successor step. The main change, of course, is that we need to understand what replaces $\RR$, $L$, and $c$. We have a hint as to what replaces $c$, namely, the sequence $\varrho$ that we ended up with after forcing with $\QQ$. We also have an idea on what to replace $\RR$ with, that would be $\power(\RR)$, and $L$ is to be replaced by $M$ itself.

We can make this much clearer if we recast the example above by replacing $\RR$ with $V_{\omega+1}$. We can also replace $c$ with a sequence of elements of $V_\omega$, of course, but this seems to needlessly complicate things. After all, the case of $\omega$ is separate anyway.

For a more uniform approach, we denote by $M_\alpha$ the $\alpha$th step in the construction, which is a model of $\ZF$ intermediate between $L$ and $L[c]$. We will also write $\varrho_\alpha$ to denote the generic sequence for $\QQ_\alpha$, the $\alpha$th forcing. So $\QQ_0$ is Cohen forcing and $\varrho_0$ is $c$ itself.

At each successor step we have $M_{\alpha+1}$ defined from $\varrho_\alpha$, which was the $M_\alpha$-generic sequence for $\QQ_\alpha$. We will assume that while $\varrho_\alpha\notin M_{\alpha+1}$, for every $\xi<\omega_{\alpha+1}$, $\varrho_\alpha\restriction A_\xi^\alpha\in M_{\alpha+1}$, where $A^\alpha_\xi$ is the $\xi$th member of the permutable family we fixed in advance. Moreover, for every $I\in[\omega_{\alpha+1}]^{<\omega_{\alpha+1}}\cap L$, $\tup{\varrho_\alpha\restriction A^\alpha_\xi\mid\xi\in I}$ is in $M_{\alpha+1}$.

We now want to define $\QQ_{\alpha+1}$ and $\varrho_{\alpha+1}$. For this we replace $\RR^{L[c\cap A_\xi]}$ that we had in the first step with $V_{\omega+\alpha+2}^{M_{\alpha+1}[\varrho_\alpha\restriction A_\xi^\alpha]}$, let us denote this as $R_\xi$ for now. The rest is more or less the same as above, relying, of course, on the recursive fact that any previous $M_\beta$ were defined much in the same way as we are defining $\QQ_{\alpha+1}$, $\varrho_{\alpha+1}$, and $M_{\alpha+2}$.

Let $\QQ_{\alpha+1}$ denote the set of approximations of $\varrho_{\alpha+1}=\tup{R_\xi\mid\xi<\omega_{\alpha+1}}$ whose domains are in $L$. We are being vague, of course, as to what counts as ``approximation'' in this context. The idea is that we may permute the different $R_\xi$ amongst themselves using a permutation of $\omega_{\alpha+1}$ which is coming from $L$.

The lemmas in the general case are exactly the same as we had before. The genericity of $\varrho_{\alpha+1}$ is proved by the same argument as \autoref{prop:genericity-1}, and while the distributivity argument is also similar, it is worth writing down.

\begin{proposition}\label{prop:dist-succ}
$M_{\alpha+1}\models\QQ_{\alpha+1}$ is ${\leq}|V_\alpha^{M_{\alpha+1}}|$-distributive. In particular, no new sets of rank $\alpha+1$ are added.
\end{proposition}
\begin{proof}
  As in \autoref{prop:dist-1}, $L[c]\models\QQ_{\alpha+1}\cong\Add(\omega_{\alpha+1},1)^L$, the latter of which is ${\leq}\aleph_\alpha$-distributive in $L[c]$. Suppose now that $\{D_x\mid x\in V_\alpha\}\in M_{\alpha+1}$ is a family of dense open subsets of $\QQ_{\alpha+1}$. By the c.c.c.\ of the Cohen forcing, and the fact we only add a single real, $V_\alpha^{M_{\alpha+1}}$ has the same cardinality as $V_\alpha^L$, which by $\GCH$ is $\aleph_\alpha$. Therefore $\bigcap\{D_x\mid x\in V_\alpha^{M_{\alpha+1}}\}$ is dense in $L[c]$ and thus in $M_{\alpha+1}$. In particular, no new subsets of $V_\alpha$ are added.
\end{proof}

Finally, we define $M_{\alpha+2}$ as the symmetric extension obtained by applying the derived group and filter using the permutable family $\cA_{\alpha+1}$. If we now consider $V_{\omega+\alpha+2}^{M_{\alpha+1}[\varrho_{\alpha+1}\restriction A^{\alpha+1}_\xi]}$, this set is in $M_{\alpha+2}$. Indeed, every proper initial segment of $\varrho_{\alpha+2}$ is in $M_{\alpha+2}$, where $\varrho_{\alpha+2}$ is the sequence of these $V_{\omega+\alpha+2}$, and it has length $\omega_{\alpha+2}$. We can now show that it is $M_{\alpha+2}$-generic, etc., and thus all shall prosper.

The successor case can be found in \cite{Karagila:Bristol} as \S4.4, as well as \S4.7 and \S4.10 for successor of limit iterands. We are not separating the successor of limit steps in this text as the idea is the same with very minor variations, so as far as outlines go, there is essentially no difference.

\subsection{Outline of limit steps}
The limit step is divided into two parts. We have to contend with the iteration at a limit step, i.e.\ the finite support limit of the previous steps, and we have to deal with the limit step itself. An observant reader will notice that we did not utilise the full power of the framework of iterating symmetric extensions until now. Indeed, as far as successor steps are concerned any automorphism coming from coordinates before $\alpha$ itself will implement the identity function on the $(\alpha+1)$th iterand, rendering it moot.

It is here, at the limit, where we need to utilise the machinery as a whole. In fact, this machinery will do most of the heavy lifting at this stage. By \autoref{prop:dist-succ} we have that the rank initial segments of the universe are stabilising, indeed for any $\beta>\alpha$ we have $V_{\omega+\alpha+1}^{M_{\alpha+1}}=V_{\omega+\alpha+1}^{M_\beta}$. The work left at this stage is making sure that the generic sequences we collected thus far are symmetrically generic, and setting up the stage for the limit step iterand. So it is a good idea to understand the limit step as a whole before proceeding to the details.

The main idea is that limit steps coalesce the information we have up to that point. Arriving to the limit is easy, as we said, the machinery of productive iterations is working for us there. But how do we proceed now? We are limited by two factors that we need to ensure continue to hold when we deal with the $\alpha$th step:
\begin{enumerate}
\item $V_{\omega+\alpha}$ is stable. That is, no new sets of low rank are added, and
\item whatever we do is coherent with the other limit steps.
\end{enumerate}

One simple way of ensuring this is by taking products of previous successor steps. This way, if $\alpha<\beta$ are two limit ordinals, then $\QQ_\alpha$ is going to be, in some sense, a rank initial segment of $\QQ_\beta$. But we can think about this from a different angle.

At each step, we gathered $V_{\omega+\xi}$ of various intermediate models, for $\xi<\alpha$, our limit ordinal. But these are smoothed out, in a sense, as we progress up the hierarchy, as each $V_{\omega+\xi+1}^{M_{\xi+2}}$ contains each $V_{\omega+\xi}^{M_{\xi+1}}$. But what if we could pick just one sequence, and remember it? In that case we are not going to add bounded sets to $V_{\omega+\alpha}$, at least not if we are being careful, and instead we only add this sequence. This idea should seem somewhat familiar to readers of all walks of set theory. After all, if we want to add a new subset to $\aleph_\omega$, it is easy to add Cohen subsets to each $\aleph_n$ first, and then choose a point from each one, creating a new cofinal sequence.\footnote{When forcing like this in the context of $\ZFC$ these cofinal sequences are added automatically, of course, but if one does a symmetric iteration, the cofinal sequences are not added. Then one can consider such a forcing in a more material sense.} Similar ideas, in one way or another, show up through Prikry-style forcings as well.

The coherence of limit steps has another very important use for the limit iterand. The following definition is very important as well.
\begin{definition}
We say that a two-step iteration $\PP\ast\dot\QQ$ is \textit{upwards homogeneous} if whenever $\tup{p,\dot q}$ and $\tup{p,\dot q'}$ are two conditions, there is an automorphism $\pi$ of $\PP$ that respects $\dot\QQ$ and such that $\pi p=p$ and $p\forces\pi\dot q=\dot q'$. In other words, we can move conditions in $\dot\QQ$ by automorphisms of $\PP$. In the context of iterating symmetric extensions we require that $\pi$ comes from the relevant automorphism group.\footnote{This requirement is quite stronger than what is actually necessary. The concept of upwards homogeneity was fully explored and analysed in \cite{UH}. We keep this definition here for historical reasons, as it was the definition which appeared in \cite{Karagila:Bristol}.}
\end{definition}

So we want for the limit step that the iteration $\PP_\alpha$ can move about conditions in $\QQ_\alpha$. If $\PP_\alpha$ is a finite support iteration, then either $\QQ_\alpha$ needs to have some finitary flavour to its conditions, or somewhere along the iteration we had to condense the conditions from finitary to infinitary. Indeed, this is the very meaning of ``coalescing'' the successor steps.

To sum up, at the limit step of the iteration we use the properties of the iteration so far to ensure that the rank initial segments of the models stabilise, and indeed that the sequence of generic sequences is symmetrically generic for the iteration. We then want to have a forcing such that the sequences which lie in the product of the successor-step generics combine to form a generic for it, here the permutable scales will come in naturally, as we are concerned with products of increasingly longer sequences modulo the bounded ideal.

We return to the context of the Bristol model's construction. We denote by $\PP_\alpha$ the iteration up to $\alpha$ and by $\QQ_\alpha$ the $\alpha$th iterand, as we did before, and for now we will assume that those iterands were defined also for limit steps. If this proves to be somewhat confusing, the section can be read twice, first assuming $\alpha=\omega$.
\begin{fact}
  For every $\alpha$, $\PP_\alpha\ast\dot\QQ_\alpha$ is upward homogeneous.
\end{fact}
We have seen this for the case of $\alpha$ being $0$ or a successor.\footnote{To be absolutely correct, we only talked about successors of $0$ or other successor ordinals, but the successor of a limit will be just the same as before.} We will see the rest of the cases in this section as we progress through it. But for now it is easier to take this as a working assumption.

\begin{proposition}
Let $\alpha$ be a limit ordinal, then $\tup{\varrho_\beta\mid\beta<\alpha}$ is symmetrically $L$-generic for $\PP_\alpha$.
\end{proposition}
This is essentially Proposition~4.4 ($\alpha=\omega$) and Proposition~4.15 ($\alpha$ an arbitrary limit ordinal) in \cite{Karagila:Bristol}. We will prove this statement in \autoref{correction:4.4}, as the original proof had a minor gap that needs to be corrected anyway.

But this means that we can understand the limit iterand fairly well now. We know that for $\beta<\alpha$, $V_{\omega+\beta+1}^{M_\beta}=V_{\omega+\beta+1}^{M_\alpha}$, and that not only we have a model of $\ZF$ which lies within $L[c]$, but that it is in fact an iteration of symmetric extensions, which means that we understand exactly the objects which lie within it and the truth value of statements about these objects from a forcing-theoretic point of view. We are now free to examine the iterand $\QQ_\alpha$.

As we reiterate time and time again, we want to ensure that no sets of rank $\omega+\alpha$ are added. In the successor steps we did that by making sure that $\QQ_\alpha$ is sufficiently distributive. For the limit step this will pose a problem. If we are to continue with our successor steps, then the $(\alpha+1)$th step needs to have a sequence of length $\omega_{\alpha+1}$. But without adding any sequences of length $\omega_\alpha$, this would mean that the sequence must have ``mostly existed'' already. So the forcing cannot be ${<}\omega_\alpha$-distributive, let alone ${\leq}|V_{\omega+\alpha}|$-distributive. In fact, if our plan is to add sequences of length $\alpha$ of sets of the form $V_{\omega+\beta}$ of some inner model, then at stages where $\cf(\alpha)=\omega$ we must have added an $\omega$-sequence.

The solution, as it turns out, is to not be distributive at all, but ensure that after applying the symmetric part of the step (rather than just the generic extension) we managed to remove any new set in $V_{\omega+\alpha}$. So even if we do not have a distributive forcing, we at least preserve the rank initial segments of the universe.

We define $\varrho_\alpha$, where $\alpha$ is a limit, as a ``copy of $\prod\SC(\omega_\alpha)^L$''. This means that for every $f\in\prod\SC(\omega_\alpha)^L$ we define $\varrho_{\alpha,f}$ to be the sequence $\tup{\varrho_{\beta+1}(f(\beta+1))\mid\beta<\alpha}$, and $\varrho_\alpha$ is defined as $\tup{\varrho_{\alpha,f}\mid f\in\prod\SC(\omega_\alpha)^L}$.

How should we define $\QQ_\alpha$, then? The idea is always ``bounded approximations'', but having the virtue of being a ``two-dimensional object'' this means that boundedness has two sides to it. Let us first deal with the one-dimensional counterparts: $\varrho_{\alpha,f}$.

If $A$ is a subset of $\alpha$, we write $\varrho_{\alpha,f}\restriction A=\tup{\varrho_{\beta+1}(f(\beta+1)\mid\beta\in A}$, and so our recursive hypothesis tell us that $\varrho_{\alpha,f}\in M_\alpha$ for every $f$, and this lets us define $\QQ_{\alpha,f}$ as the approximations of $\varrho_{\alpha,f}$. More rigorously, recall that we have defined the symmetric iteration $\PP_\alpha$, along with the direct limit of the generic semidirect products, $\cG_\alpha$. Then $\dot\QQ_{\alpha,f}$ is the forcing ordered by reverse inclusion on the interpretation of \[\left\{\gaut{\vec\pi}\dot\varrho_{\alpha,f}\restriction A\middd\sup A<\alpha, \gaut{\vec\pi}\in\cG_\alpha\right\}^\bullet.\]

For $E\subseteq\prod\SC(\omega_\alpha)$ we may now define $\varrho_\alpha\restriction E=\tup{\varrho_{\alpha,f}\mid f\in E}$, and so if $E$ is bounded in the product, i.e.\ there is $f\in\prod\SC(\omega_\alpha)$ such that for all $g\in E$ and for all $\beta<\alpha$, $g(\beta+1)<f(\beta+1)$, and $A$ is a bounded subset of $\alpha$, our conditions are going to be approximations of $\varrho_\alpha$, up to permutations of course, of the form $\tup{\varrho_{\alpha,f}\restriction A\mid f\in E}$, which we denote by $\varrho_\alpha\restriction(E,A)$. And so $\QQ_\alpha$ is given by the name \[\left\{\gaut{\vec\pi}\dot\varrho_\alpha\restriction(E,A)\middd E,A\text{ are bounded and }\gaut{\vec\pi}\in\cG_\alpha\right\}^\bullet.\]

\begin{proposition}
$\PP_\alpha\ast\dot\QQ_\alpha$ is upwards homogeneous.
\end{proposition}
We do not prove this statement here, but the idea is to simply utilise the upwards homogeneity of the previous steps and ``correct'' the coordinates one by one. One might ask how do we deal with the case where $\alpha>\omega$, as a condition has seemingly infinitely many non-trivial coordinates. And the answer, as we repeatedly mention here, is utilising the previous limit cases where we condense this infinite amount of information, also in the form of $\sG_\alpha$ being a subgroup of the full support product $\prod_{\beta<\alpha}\sG_{\beta+1}$.\footnote{This is not the same as the group $\cG_\alpha$ which will only include bounded sequences, which means that it will only implement the identity.} The complete proof can be found as Propositions 4.5 (for $\alpha=\omega$) and 4.15 in \cite{Karagila:Bristol}. The action of $\sG_\alpha$ is coordinatewise, and it is important to stress at this point that we have this action where the initial segments of $\varrho_\alpha$ are ``actual objects in $M_\alpha$'', so this is not applying automorphisms of a forcing, but rather applying permutations of each $\omega_{\beta+1}$.

\begin{proposition}
$\varrho_\alpha$ is $M_\alpha$-generic for $\QQ_\alpha$.
\end{proposition}
This again follows the same pattern as the successor steps, although here there is a notable complication in the case where $\alpha>\omega$. We will outline the proof.
\begin{proof}
  Let $\dot D\in\IS_\alpha$ be a name for a dense open subset of $\dot\QQ_\alpha$, and let $\vec H$ be an excellent support witnessing that $\dot D\in\IS_\alpha$. For each $\beta<\alpha$, $H_\beta$ is a group of either the form $\fix(\cB_\beta)$ when $\beta$ is $0$ or successor, or $K_{\eta_\beta,f_\beta}$ when $\beta$ is itself a limit. We can use these to define bounded sets $E$ and $A$ such that whenever $\varrho_\alpha\restriction(E,A)$ is extended to a condition in $\dot D$, we may permute this extension using the upwards homogeneity without changing $\dot D$.

  In the case where $\alpha=\omega$, the set $E$ is simple. For each $n<\omega$ we let $\cB_n$ be a disjoint approximation such that $\fix(\cB_n)=H_n$, then $E=\prod_{n<\omega}\dom\cB_n$, where $\dom\cB_n=I$ such that $\cB_n=\{B_\xi\mid\xi\in I\}$.\footnote{In the original paper the proof of $\alpha=\omega$ is Lemma~4.7, and there is a minor mistake in the proof: $B_n$ should be $\dom\cB_{n-1}$. Lemma~4.20, which is the general claim, has a correct proof.\label{foot:4.7}} Let $A$ be the set $\{n<\omega\mid H_n\neq\sG_n\}$, which is also bounded by the virtue of $\vec H$ being excellent, and we have ourselves the condition $\varrho_\omega\restriction(E,A)$.

  In the case where $\alpha>\omega$ we need to take into consideration some limit point $\delta<\omega$, such that either $\delta+\omega=\alpha$, or $\delta$ is the smallest limit ordinal such that $\vec H$ is trivial above $\delta$. We call such $\delta$ the condensation point of $\vec H$. So above $\delta$ there can be at most finitely many non-trivial coordinates, and only in the case where $\delta=\alpha+\omega$. We then let $E_\beta$ for $\beta<\delta$ be decided by $H_\delta=K_{\eta_\delta,f_\delta}$ by taking $E_\beta=f_\delta(\beta)$, and above $\delta$ we do as we did in the case of $\omega$.

  Now we may proceed as before, using the fact that any extension of $\varrho_\alpha\restriction(E,A)$ would have the form $\gaut{\vec\pi}\varrho_\alpha\restriction(E',A')$, where $\pi_\beta$ will necessarily have $\pi_\beta\restriction E_\beta=\id$, so we may find the needed automorphisms in $\vec H$ to complete the proof as in the successor case.
\end{proof}

Now we only need to take care of the preservation of $V_{\omega+\alpha}^{M_\alpha}$. Indeed, $M_\alpha[\varrho_\alpha]$ contains the sequences $\varrho_{\beta+1}$ for $\beta<\alpha$, all of which have rank smaller than $\omega+\alpha$. Luckily, the symmetries of $\QQ_\alpha$ will help us get rid of these unwelcomed sets.\footnote{This also highlights the importance of the assumptions in \autoref{thm:preservation} being only about the symmetric extensions having the same $V_\alpha$.}

\begin{definition}
Working in $M_\alpha$, if $\dot x$ is a $\QQ_\alpha$-name we say that it is \textit{bounded by $f\in\prod\SC(\omega_\alpha)^L$} if whenever $\gaut{\vec\pi}\varrho_\alpha\restriction(E,A)\forces\dot y\in\dot x$, then we may replace $E$ by $E\cap f\da$, where $f\da=\{g\in\prod\SC(\omega_\alpha)^L\mid\forall\beta, g(\beta)<f(\beta)\}$. Similarly, if $\beta<\alpha$ we say that $\dot x$ is \textit{bounded by $\beta$} if we can replace $A$ by $A\cap\beta$.
\end{definition}

\begin{theorem}
  Suppose that $\dot x\in M_\alpha$ is a $\QQ_\alpha$-name such that every name appearing in $\dot x$ is $\check y$ for some $y\in M_\alpha$.
  \begin{enumerate}
  \item If $\dot x\in\HS_{\sF_\alpha}$ and $K_{\eta,f}\subseteq\sym(\dot x)$, then $\dot x$ is bounded by $f$.
  \item If $\forces\rank(\dot x)<\check\omega+\check\alpha$, then $\dot x$ is bounded by some $\beta<\alpha$.
  \end{enumerate}
\end{theorem}

Combining these we have that if $\dot x$ is a symmetric name for a set of small rank, then all of its elements are decided by conditions with a uniform bound, meaning $\dot x$ is equal to an object in $M_\alpha$. The proof of (1) is the standard homogeneity argument, and we have used it before in \autoref{thm:example-ccc}, so we will only outline the proof of (2).

\begin{proof}
  Suppose that $\dot x\in\IS_{\alpha+1}$, we denote by $[\dot x]$ the projection of the name to $\IS_\alpha$. That is, $[\dot x]$ is a name in $\IS_\alpha$ which is interpreted in $M_\alpha$ as a name in $\HS_{\sF_\alpha}^{M_\alpha}$. By the assumption that each name appearing in $\dot x$ is of the form $\check y$ for some $y\in M_\alpha$, we may assume that $[\check y]$, which appears in $[\dot x]$ (in the broad sense of the term) is a name in $\IS_\beta$ for $\beta$ such that $\omega+\beta$ is an upper bound on the forced rank of $\dot x$.

  Let $\beta<\alpha$ be large enough such that $\vec H$ is trivial above $\beta$, where $\vec H$ witnesses that $[\dot x]\in\IS_\alpha$. Now we can use automorphisms which only move coordinates above $\beta$ to move any names of conditions, $\gaut{\vec\pi}\varrho_\alpha\restriction(E,A)$, by changing their ``content above $\beta$'' to any value. Thus, we may conclude that we may reduce $A$ to $A\cap\beta$.
\end{proof}
The theorem is proved as Lemma~4.10 and Lemma~4.27 in \cite{Karagila:Bristol}. This almost completes our decoding apparatus. We only need to worry about the $\varrho_{\alpha+1}$ now.

We define $R_\xi$ for $\xi<\omega_{\alpha+1}$ to be $V_{\omega+\alpha+1}^{M_\alpha[\varrho_{\alpha,f^\alpha_\xi}]}$. That is, we use $\varrho_{\alpha,f^\alpha_\xi}$ as the ``guide'' for a new sequence in $V_{\omega+\alpha+1}$. Those who kept track can guess now that $\varrho_{\alpha+1}$ is $\tup{R_\xi\mid\xi<\omega_{\alpha+1}}$. We utilise the fact that the scale is permutable to ensure that any bounded part will be in $M_{\alpha+1}$, as well as the rest of the permutability apparatus to ensure the upwards homogeneity. This is also the point where we see why the definition of $\QQ_\alpha$ works in general, despite $\PP_\alpha$ being a finite support iteration. The $\varrho_{\alpha,f}$ are initial segments of those $\varrho_{\beta,f^*}$ that come in the future, and even if $\cf(\alpha)>\omega$, we still end up with what we wanted to have.

With this we finish the discussion on the decoding apparatus. This is the main technical part of the construction. It is our sword and shield in our journey downwards. Now that we have that, we may venture deeper into the Hadean adventure that is the Bristol model.

\section{Three views of the Bristol model}\label{sect:cerberus}
Much like Cerberus, the construction of the Bristol model can be seen as a three-headed dog, guarding the realm of the underworld of models of $\ZF$. The three heads of Cerberus represent the three causes of strife: nature, cause, and accident. The three heads of the Bristol model can be seen as representing three typical approaches to its construction: nature, cause, and accident. All lead us to the same construction, and indeed when getting down to brass tacks, the details become suspiciously similar in each approach. But the presentation of one may be more appealing to some readers over another.

Nature is the way in which the original group in Bristol went about to define the model: defining the von Neumann hierarchy by hand, and defining a model $L(X)$ where $X$ is a class of sets in the Cohen extension. Cause is the way in which \cite{Karagila:Bristol} presents the construction: defining an iteration of symmetric extensions, and finding a symmetrically generic filter at each step of the way, thus constructing an iteration and the von Neumann hierarchy of the Bristol model in tandem. Finally, Accident is the way in which we define a productive iteration of symmetric extensions, we study this iteration in an abstract manner, and then we find that by ``complete accident'' we can find all the symmetrically generic filters inside the Cohen extension.

Fix a Bristol sequence, a permutable family $\cA_\alpha=\{A^\alpha_\xi\mid\xi<\omega_{\alpha+1}\}$ for $\alpha=0$ or a successor, and a permutable scale $F_\alpha=\{f^\alpha_\xi\mid\xi<\omega_{\alpha+1}\}$ for $\alpha$ limit. We will define $M$, the Bristol model, in three different, yet equivalent ways. We will argue that it is a model of $\ZF+\forall x(V\neq L(x))$.

\subsection{Nature}
In here we will define the Bristol model one step at a time by defining its von Neumann hierarchy in $L[c]$. For this purpose it would be easier at times to use an increasing sequence modulo bounded sets, rather than the permutable families. Since the two are equivalent, we let $\{T^\alpha_\xi\mid\xi<\omega_{\alpha+1}\}$ denote a sequence obtained from $\cA_\alpha$.

We define the $\varrho_\alpha$ and $V_{\omega+\alpha}^M$ in tandem, and we will omit the $M$ from the superscript where possible, as the definitions will be complicated enough. Let $\varrho_0=c$ and, as $V_\omega^M$ is just $V_\omega^L=L_\omega$, it is defined. Suppose that for $\alpha$, $\varrho_\alpha$ and $V_{\omega+\alpha}$ were defined.

If $\alpha$ is $0$ or a successor ordinal, define
\begin{align}
V_{\omega+\alpha+1}=& \bigcup_{\xi<\omega_{\alpha+1}} V_{\omega+\alpha+1}^{L(V_{\omega+\alpha},\varrho_\alpha\restriction T^\alpha_\xi)},\\
\varrho_{\alpha+1}=& \tupp{V_{\omega+\alpha+1}^{L(V_{\omega+\alpha},\varrho_\alpha\restriction A^\alpha_\xi)}\middd\xi<\omega_{\alpha+1}}.
\end{align}
If $\alpha$ is a limit ordinal, we define $V_\alpha=\bigcup_{\beta<\alpha} V_{\omega+\beta+1}$, and we define $\varrho_\alpha=\prod_{\beta<\alpha}\varrho_{\beta+1}$. And as before we write $\varrho_{\alpha,f}$ to indicate the ``thread'' of the function $f$ in this product. To define the next step we need an analogue of the $T^\alpha_\xi$, we write $\overline{\varrho_{\alpha,f}}$ to denote the sequence of $V_{\omega+\beta+1}^{L(V_\omega+\beta,\varrho_\beta\restriction T^\beta_{f(\beta)})}$ for $\beta<\alpha$.

And we now define
\begin{align}
  V_{\omega+\alpha+1}=& \bigcup_{\xi<\omega_{\alpha+1}} V_{\omega+\alpha+1}^{L(V_{\omega+\alpha},\overline{\varrho_{\alpha,f^\alpha_\xi}})},\\
  \varrho_{\alpha+1}=& \tupp{V_{\omega+\alpha+1}^{L(V_{\omega+\alpha+1},\varrho_{\alpha,f^\alpha_\xi})}\middd\xi<\omega_{\alpha+1}}.
\end{align}
Finally, $M=\bigcup_{\alpha\in\Ord} V_{\omega+\alpha+1}^M$. The handwritten notes passed on to us by some of the members of the Bristol group indicate that the original line of thought was about $\RR^M$ and $\power(\RR)^M$, rather than $V_{\omega+1}$ and $V_{\omega+2}$. The arguments, moreover, as to why $V_{\omega+1}=V_{\omega+1}^{L(V_{\omega+1},\varrho_1\restriction T^\alpha_\xi)}$, i.e.\ why no reals are added when defining $\power(\RR)^M$, was not written down in these notes. Instead it merely suggests ``condensation''.

While restoring the original arguments is certainly beyond us, these would be equivalent, more or less, to the arguments we presented at the first step of \autoref{sect:decoding}.
\subsection{Cause}
In here we will define the Bristol model in tandem: define a symmetric extension, find it a generic in $L[c]$, define a symmetric extension again, repeat ad ordinalum.

This definition was used in \cite{Karagila:Bristol}, and you can very clearly see that this is the definition that we have in mind, as it very much dominates the approach given in \autoref{sect:decoding}. As such, we really have done all the work ahead of time.

We define $M_\alpha,\varrho_\alpha$ as in the decoding apparatus, i.e.\ as the generic objects for the symmetric iteration. We now define $M$ simply as $\bigcup_{\alpha\in\Ord}M_\alpha=\bigcup_{\alpha\in\Ord}V_{\omega+\alpha}^{M_{\alpha+1}}$.
\subsection{Accident}
In here we will first define a class-forcing, and then argue that we can just happen to find symmetrically generic filters in $L[c]$.

Despite being the guiding view on the construction of the Bristol model, the actual argument for $M\models\ZF$ in \cite{Karagila:Bristol} is the one rising from this approach, as it is less ``ad-hoc'' and more structural and general. This is also the reason why the proof of the distributivity of successor steps in the original paper was proved directly, rather than the approach used in this paper, which is more in line with ``Cause''.

We first define the iteration. Let $\PP_0=\{\1\}$ and $\QQ_0=\Add(\omega,1)$, $\sG_0$ and $\sF_0$ are the derived group and filter. Finally, $\dot\varrho_0$ is the canonical name for the Cohen real.

Suppose that $\PP_\alpha$ is the finite support iteration of the symmetric extensions defined so far. In the case where $\alpha=\beta+1$, we define $\dot\QQ_\alpha$ as the name \[\left\{\gaut{\vec\pi}\tupp{\dot V_{\omega+\alpha}^{V[\varrho_\beta\restriction A^\beta_\xi]}\middd\xi\in I}^\bullet\middd\gaut{\vec\pi}\in\cG_\alpha, I\in[\omega_\alpha]^{<\omega_\alpha}\right\}^\bullet.\]
Where $\dot V_{\omega+\alpha}^{V[x]}$ denotes the name of the rank initial segment of the extension of $\IS_\beta$ by the set $x$. We then define $\dot\varrho_\alpha$ as the name for the generic of $\dot\QQ_\alpha$.

If $\alpha$ is a limit ordinal, we define $\PP_\alpha$ as the finite support iteration. We next define $\dot\QQ_\alpha$ in the same spirit. For $f\in\prod\SC(\omega_{\alpha+1})$ we define \[\dot\QQ_{\alpha,f}=\{\gaut{\vec\pi}\tup{\dot\varrho_{\beta+1}(f(\beta+1))\mid\beta\in A}\mid\gaut{\vec\pi}\in\cG_\alpha, \sup A<\alpha\}^\bullet,\]
and we then define $\dot\QQ_\alpha$ by adding the additional dimension of a bounded subset of $\prod\SC(\omega_\alpha)$. As before, $\sG_\alpha$ and $\sF_\alpha$ are the derived group and filter. Finally, $\dot\varrho_\alpha$ is the name of the generic filter. Moving on, the definition $\dot\QQ_{\alpha+1}$ and $\dot\varrho_{\alpha+1}$ is in line with what we have done so far.

The properties of the decoding apparatus imply the distributivity of each iterand. We need to be a bit more careful here, as we are not allowed to argue in $L[c]$, instead we carry on a recursive hypothesis that for successor steps $\PP_\alpha$ satisfies a chain condition that allows us to prove that $\dot\QQ_\alpha$ will be isomorphic to $\Add(\omega_\alpha,1)^L$.

The conditions of \autoref{thm:preservation} are therefore satisfied. This implies that if $G$ is any symmetrically $L$-generic for $\PP$, the class-length iteration, then $\IS^G\models\ZF$, at the very least.

By recursion we can now show that each $\dot\varrho_\alpha$ has a symmetric interpretation inside $L[c]$, and therefore we may find an interpretation of the model which is intermediate between $L$ and $L[c]$.
\subsection{The basic properties of the Bristol model}
\begin{proposition}
$M\models\ZF$.
\end{proposition}
\begin{proof}
We have two slightly different proofs here, the first we mentioned in the ``Accident'' approach, utilising \autoref{thm:preservation}.\footnote{This is mentioned after \autoref{thm:preservation}, and proved as Theorem~9.4 in \cite{Karagila:Iterations}.} The second approach works well for ``Nature'' and ``Cause''.

Since $M$ is very clearly a transitive subclass of $L[c]$ that contains all the ordinals, it is enough to verify that it is closed under G\"odel operations and that it is almost universal to conclude that it is a model of $\ZF$. The first part is easy: if $x,y\in M$, there is some $M_\alpha$, or some large enough $V_{\omega+\alpha}$ in the ``Nature'' approach, such that $x,y\in M_\alpha$ and therefore $\{x,y\}$, $x\times y$, etc.\ are all in $M_\alpha$ and therefore in $M$.

  The second part is the fact mentioned at the end of ``Nature'' about condensation; or in the case of ``Cause'' this follows from the fact that for each $\alpha$, $V_{\omega+\alpha}^{M_\alpha}=V_{\omega+\alpha}^{M_\beta}$ for all $\alpha\leq\beta$, and therefore the model is almost universal.
\end{proof}
\begin{proposition}\label{prop:non-constructible}
$M\models\forall x(V\neq L(x))$.
\end{proposition}
\begin{proof}
If $x\in M$, then there is some $\alpha$ such that $x\in M_\alpha$, and therefore $L(x)\subseteq M_\alpha$. But since $M_\alpha\subsetneq M_{\alpha+1}\subseteq M$, we have that $L(x)\neq M$. In the ``Nature'' approach we need to be slightly more careful, as we have to verify that if $x\in V_\alpha$, then $V_{\alpha+1}\notin L(x)$. One way of doing this would be to argue that $\varrho_\alpha\restriction A^\alpha_\xi$ are all generic over $L(V_\alpha)$, and therefore cannot be elements of this model.
\end{proof}
We will see other ways of deducing \autoref{prop:non-constructible} later on, in ways that will be significantly more informative and more general.

\section{Errata to ``the Abyss''}
Here we correct some minor gaps and mistakes in the original paper. These mistakes repeat, once with the cases of $\alpha=0$ and $\alpha=\beta+2$, and once with the cases of $\alpha=\omega$ and a general limit ordinal.

We repeat the remark we made in \autoref{foot:4.7}: in \cite{Karagila:Bristol} the proof of Lemma~4.7, dealing with the genericity of $\varrho_\omega$, contains a typo, whereas $B_n$ is defined as $\bigcup\cB_n$, and it should have been defined as $\dom\cB_{n-1}$. This is somewhat inconsequential, as the more general proof where $\alpha$ is a limit, in Lemma~4.20, is written correctly.

We also point out that there is a minor mistake in the construction's induction hypothesis, where the requirement that $\PP_\alpha$ satisfies the $\aleph_\alpha$-c.c.\ will not hold for limit ordinals, only for $0$ and successor ordinals. We can modify this by defining a notion of ``symmetric chain condition'',\footnote{For example, $\tup{\PP,\sG,\sF}$ has $\kappa$-s.c.c.\ if every symmetrically dense open set contains a predense set of size ${<}\kappa$.} but this is an unnecessary complication. We simply do not use the chain condition assumption for limit iterands.

\subsection{Corrections to Lemma 2.12/4.1}\label{correction:2.12}
Lemma~2.12 is read in the context of the first steps. Namely, $\PP$ is Cohen forcing, $\dot\QQ$ is the name of the forcing adding $\varrho_1$, denoted in that lemma by $\sigma$. As such $\HS$ is the class of hereditarily symmetric names defined in the very first step.
\begin{lemma*}[\cite{Karagila:Bristol}, Lemma~2.12]
  Suppose that $\dot D\in\HS$ and $p\forces\dot D\subseteq\dot\QQ$ is a dense open set. There is some $\eta<\omega_1$ such that for every $\pi$ and $A$ such that $p\forces\pi\dot\sigma_A\in\dot D$ and $\eta\subseteq A$, if $\tau\dot\sigma_A$ is a condition such that $p\forces\pi\dot\sigma_\eta=\tau\dot\sigma_\eta$, then $p\forces\tau\dot\sigma_A\in\dot D$ as well.
\end{lemma*}
In the paper this lemma was used twice. The first consequence is the genericity of $\varrho_1$,\footnote{Or $\sigma$, in that context.} and the second is to show that $\dot\QQ$ is $\sigma$-distributive. The proof is very similar to the proof of \autoref{prop:genericity-1}. Nevertheless, when we have two permutations $\pi$ and $\tau$, we want to move $\tau$ so that it agrees with $\pi$.

For this we need not only that $\iota(\tau)\restriction A=\iota(\pi)\restriction A$, but also that their inverses agree on $A$. In the case of the genericity of $\varrho_1$ we take $\pi=\id$, in which case this property holds trivially. Indeed, this is the very strategy of the proof of \autoref{prop:genericity-1}.

To prove that $\dot\QQ$ is $\sigma$-distributive we interpret the lemma above as saying that for a condition $p\in\PP$, there is an ordinal $\eta$, such that deciding whether a condition from $\dot\QQ$, whose domain is at least $\eta$, lies in $\dot D$ depends only on the permutation $\pi$. We then utilise the fact $\PP$ is c.c.c.\ to show that this $\eta$ can be bound uniformly depending on $\dot D$. Now, if $\dot D_n$ are names for a sequence of dense open sets, we can uniformly bound all of them by some $\eta$. Now if we take any condition whose domain is at least this $\eta$, we may extend it into each $\dot D_n$, but this means that such extension lies, in fact, in all the $\dot D_n$ simultaneously, and therefore the intersection is also dense.

Once we add the condition that $\iota(\pi)^{-1}\restriction A=\iota(\tau)^{-1}\restriction A$, the proof as written in \cite{Karagila:Bristol} follows through. Alternatively, and perhaps more wisely, for proving the distributivity one should rely on the absoluteness proof that we used in this manuscript, which is also mentioned in the ``Abyss''.

Lemma~4.1 is precisely the same, in the context of successor iterations. There is no need to repeat that which has been said.
\subsection{Corrections to Proposition 4.4/4.15}\label{correction:4.4}
These two propositions are dealing with limit iterations. They show that $\tup{\varrho_\beta\mid\beta<\alpha}$ is symmetrically generic for $\PP_\alpha$, with Proposition~4.4 dealing with the case $\alpha=\omega$, which was treated separately in the paper.
\begin{proposition*}[\cite{Karagila:Bristol}, Proposition~4.4]
Suppose that $D\subseteq\PP_\omega$ is a symmetrically dense open set, then there is a sequence $\tup{\beta_n\mid n<\omega}$ such that $\tup{\dot\varrho_n\restriction\beta_n\mid n<\omega}\in D$.
\end{proposition*}
In the paper the proof takes a symmetrically dense open set $D$, and an excellent support $\vec H$ witnessing that. We then generate some sequence of domains such that when we extend the condition $\tup{\varrho_\beta\restriction X_\beta\mid\beta<\alpha}$ into $D$, then we can ``correct'' it using automorphisms which lie in $\vec H$. Namely, each coordinate of $\vec H$ is of the form $\fix(\cB_\beta)$, and we take $X_\beta$ to be $\sup\dom\cB_\beta+1$, and therefore our starting condition is $\tup{\varrho_\beta\restriction X_\beta\mid\beta<\alpha}$.

The idea is fine, and it is somehow intuitively clear what should happen. However the proof described above, taken from the ``Abyss'', will not work. The first hint is obvious: $\varrho_\beta$ has domain $\omega_\beta$, and $X_\beta$ is a bounded subset of $\omega_{\beta+1}$. We can resolve this issue by considering $\varrho_{\beta+1}\restriction X_\beta$, but this raises the obvious question, what should we do with limit steps and with $\varrho_0$?

Let us prove this proposition in its general form.

\begin{proposition}[\cite{Karagila:Bristol}, Proposition~4.15]
Let $\alpha$ be a limit ordinal, then the sequence $\tup{\varrho_\beta\mid\beta<\alpha}$ is symmetrically $L$-generic for $\PP_\alpha$.
\end{proposition}
\begin{proof}
Let $D$ be a symmetrically dense open set. Our goal is to find a condition of the form $\tup{\varrho_\beta\restriction X_\beta\mid\beta<\alpha}$, for appropriate $X_\beta$, in $D$.

  Let $\vec H$ denote the excellent support witnessing that $D$ is a symmetrically dense open set. For $\beta<\alpha$ let $\xi_{\beta+1}$ the ordinal defined by either,
  \begin{enumerate}
  \item when $\beta$ is $0$ or a successor let $\cB_\beta$ be such that $H_\beta=\fix(\cB_\beta)$, and define $\xi_{\beta+1}=\sup\dom\cB_\beta$;
  \item when $\beta$ is a limit let $\xi_{\beta+1}$ be an ordinal $\xi$ such that for some $f\in\prod\SC(\omega_\beta)$, $H_\beta=K_{\xi,f_\beta}$;
  \item when $H_\beta=\sG_\beta$, set $\xi_{\beta+1}=0$.
  \end{enumerate}

  For a limit ordinal $\beta$, let $E_\beta\subseteq\prod\SC(\omega_\beta)$ be defined by $\prod_{\gamma<\beta}(\xi_{\gamma+1}+1)$, and let $A_\beta$ denote $\sup\{\gamma<\beta\mid H_{\gamma+1}\neq\sG_{\gamma+1}\}$; if $\beta$ is a limit ordinal such that $H_\gamma=\sG_\gamma$ for all $\gamma\geq\beta$, set $E_\beta=A_\beta=\varnothing$. Now define $Y_\beta$ as $0$ for $\beta=0$, $\xi_\beta$ for successors, and $(E_\beta,A_\beta)$ for limit ordinals.

  The goal is to find the ``domain'' over which $\vec H$ must be the identity, and use that to start our journey into $D$. Let $p$ be the condition $\tup{\varrho_\beta\restriction Y_\beta\mid\beta<\alpha}$. By its very definition, $p$ is not moved by any $\vec\pi\in\vec H$. Using the density of $D$ we can now extend $p$ to a condition $p_0\in D$. Moreover, since $D$ is in $L$, we may assume that the first coordinate of $p_0$ agrees with $\varrho_0$, i.e.\ the Cohen real $c$, and by shrinking $H_0$ if necessary, we may assume that $p_0(0)$ is not moved by automorphisms in $H_0$.

  If $p_0$ is not compatible with $\tup{\varrho_\beta\mid\beta<\alpha}$, let $\beta$ be the least coordinate witnessing that. By the choice of $Y_\beta$ we can find $\vec\pi\in H\restriction\beta$, and if $\beta$ is not a limit ordinal then we may assume $\vec\pi$ has a single non-trivial coordinate too, such that by taking $p_1=\gaut{\vec\pi}p_0$, the following hold:
  \begin{enumerate}
  \item $\supp(p_0)=\supp(p_1)$,
  \item $p_0\restriction\beta=p_1\restriction\beta$,
  \item $p_1(\beta)$ is indeed compatible with $\varrho_\beta$.
  \end{enumerate}
  The reason we can do that is exactly upwards homogeneity combined with our choice of $p$, which $p_0$ extended. As we did not increase the non-trivial coordinates when moving from $p_0$ to $p_1$, we may proceed by recursion and after finitely many steps the process must halt with some $p_n\in D$.
\end{proof}

Indeed, the main point is the fact that we may assume that the first coordinate, for which true genericity is already assumed, is compatible with the Cohen real, and then we can modify any further coordinates recursively using upward homogeneity, combined with the fact that we only need to change things outside of the domains we found, which means that the needed automorphisms can be found in $\vec H$.
\section{Removing the constructibility assumptions}
Now that we have constructed the Bristol model, and we have a good idea about how the construction works, we can ask the obvious question: do we really need $V=L$ in the ground model? The answer, of course, is not really.

We have merely used three assumptions: $\GCH$, $\square^*_\lambda$ for singular $\lambda$, and global choice for fixing the Bristol sequence. Of those three, the last one can be easily dispensed, and we will discuss this in \autoref{sect:gaps}. So we are left with only two assumptions which are compatible with a large range of models, including all known inner models from large cardinals below a subcompact.

This raises an interesting question, of course. What kind of large cardinals can the Bristol model accommodate, assuming they existed in the ground model?
\begin{proposition}\label{prop:0-sets}
If $A$ is a set of ordinals in the Bristol model, then there is a real number $r$ such that $A\in V[r]$, and moreover $r\in V$ or $r$ is Cohen over $V$.
\end{proposition}
\begin{proof}
Note that $V[A]$ is a model of $\ZFC$ intermediate to $V$ and $V[c]$, and therefore $V[A]$ is equivalent to some $V[r]$.
\end{proof}

This leads to an immediate corollary: if $\kappa$ is a large cardinal defined by the existence (or lack thereof) of sets of ordinals, and this largeness is preserved by adding a Cohen real, then $\kappa$ remains large in the Bristol model.

For example, if $\kappa$ is Mahlo, then the stationary set of regular cardinals remain stationary in $M$, and thus $\kappa$ remains Mahlo. If we define a weakly compact cardinal by stating that every colouring of $[\kappa]^2$ in $2$ colours has a homogeneous subset, this too is given by sets of ordinals, and so it continues to hold in $M$.

We will see in \autoref{sect:choice} that this can be extended from sets of ordinals to sets of sets of ordinals, and so on, as long as we iterate power sets less than $\kappa$ times, the largeness remains. So for example, any measure on a ground model measurable cardinal will have a unique extension in the Bristol model. Strong cardinals, defined by extenders, are also preserved.

\subsection{Limitations}
Despite this handsome accommodation of large cardinal assumptions in the ground model, as well as in the Bristol model itself, we can put a stop to this. Indeed, if $\kappa$ is a supercompact cardinal, then there is a notable shortage of $\square^*_\lambda$ sequences for singular $\lambda$ such that $\cf(\lambda)<\kappa<\lambda$. We may ask ourselves, perhaps we can salvage permutable scales without having $\square^*_\lambda$?

\begin{theorem}[The Bristol group]
Suppose that $\kappa$ is a supercompact cardinal and $\lambda>\kappa>\cf(\lambda)$. Then there are no permutable scales on $\prod\SC(\lambda)$.
\end{theorem}
Recall that we are still working under the assumption of $\GCH$. Removing it will require adding $2^\lambda<2^{\lambda^+}$, which itself is a harmless assumption as it holds for all strong limit cardinals above $\kappa$.
\begin{proof}
  Let $F=\{f_\alpha\mid\alpha<\lambda^+\}$ be a scale in $\prod\SC(\lambda)$, and let $\pi$ be a permutation of $\lambda^+$, not necessarily bounded, which is not implemented by any sequence of permutations, $\vec\pi$. Pick $j\colon V\to W$ an elementary embedding with critical point $\kappa$ witnessing that $\kappa$ is at least $\lambda^+$-supercompact, so in particular $j(\kappa)>\lambda^+$. We denote by $G=\{g_\alpha\mid\alpha<j(\lambda^+)\}$ the scale $j(F)$, which is a scale in $j(\prod\SC(\lambda))$ which is equal to $\prod\SC(j(\lambda))^M$.

  Let $\nu=\sup j``\lambda^+$, then $\nu$ is closed under $j(\pi)$, so we can let $\tau$ denote $j(\pi)\restriction\nu$, which is a bounded permutation of $j(\lambda^+)$.

  Suppose that there was a sequence of permutations $\vec\tau$ that implemented $\tau$, then for $\alpha<\lambda^+$ and $\mu<\lambda$, we have $g_{j(\alpha)}(j(\mu^+))=j(f_\alpha(\mu^+))$ and $\tau(j(\alpha))=j(\pi(\alpha))$. Combining this with the assumption that $\vec\tau$ implements $\tau$, we get that \[\tau_{j(\mu^+)}(j(f_\alpha(\mu^+)))=j(f_{\pi(\alpha)}(\mu^+)).\]

  We use this to define $\pi_{\mu^+}\colon\mu^+\to\mu^+$. If $\tau_{j(\mu^+)}(j(\xi))=j(\zeta)$ for some $\zeta$, define $\pi_{\mu^+}(\xi)=\zeta$; otherwise $\pi_{\mu^+}(\xi)=\xi$. This is well-defined, since $\tau_{j(\mu^+)}$ is itself a permutation of $j(\mu^+)$, and so $j(\zeta)<j(\mu^+)$. This is also a permutation by similar reasons.

  Finally, we claim that $\vec\pi$ implements $\pi$. Fix $\alpha<\lambda^+$, then for any large enough $\mu^+<\lambda$, $\tau_{j(\mu^+)}(j(f_\alpha(\mu^+)))=j(f_{\pi(\alpha)}(\mu^+))$, which means that we defined $\pi_{\mu^+}$ to be exactly $f_{\pi(\alpha)}(\mu^+)$.

  We have shown that if $F$ can be used to implement all bounded permutations, then it can be used to implement all permutations, but that is definitely impossible on grounds of cardinal arithmetic.
\end{proof}
\section{Gaps in the Multiverse}\label{sect:gaps}
The Bristol model, as we said at first, is a particularly striking counterexample to our understanding of intermediate models when the axiom of choice is not assumed. Working in $V$'s meta-theory, we can consider the collection of all intermediate models between $V$ and $V[c]$. We remarked before, those models which satisfy $\ZFC$ are exactly $V$ itself or Cohen extensions of the form $V[r]$ for some $r\in\RR^{V[c]}$.

What about models of $\ZF$? We have, of course, symmetric extensions. These were studied extensively by Grigorieff in \cite{Grigorieff:1975}, and later by Usuba in \cite{Usuba:2019}. Grigorieff proved that if $M$ is a model such that $V\subseteq M\subseteq V[G]$, then $M$ is a symmetric extension given by the same forcing used to add $G$ if and only if $M=(\HOD_{V\cup x})^{V[G]}$ for some set $x\in V[G]$. Usuba extended this and showed that in general, $M$ is a symmetric extension of $V$ if and only if $M=V(x)$ for some set $x$.

The construction of the Bristol model shows that there is indeed a difference between the two results. Indeed, by counting the number of possible automorphism groups and filters of groups we can see that there cannot be more than $\aleph_3$ distinct symmetric extensions of $\Add(\omega,1)$ over $L$.\footnote{We can improve this counting argument to show no more than $\aleph_2$, actually.} Yet, the construction of the Bristol model goes through a proper class of steps, and even if we did not formally prove that each separate one is of the form $L(x)$, we may take $L(V_\alpha^M)$ as our models. By Usuba's result, these are all symmetric extensions of $L$, but of course most of them are not symmetric extensions where the forcing used is the Cohen forcing.

So when we consider the multiverse of symmetric extensions of $L$, even those that are landlocked inside $L[c]$, we seem to have two different options. But we want to study all intermediate models, and these include the Bristol model, which is very much not a symmetric extension of $L$, by any means, as Usuba's result indicate. So what can we say about the multiverse of $\ZF$ models?

First, let us ask, how many Bristol models are there? Of course, there is the canonical one, given by the $<_L$-minimal Bristol sequence. But there are certainly more. Simply by removing some of the sets or functions in any given point in the Bristol sequence we invariably create a new Bristol model.
\begin{proposition}
Assuming $\GCH$ and that $\square^*_\lambda$ holds for all singular $\lambda$, there is a class forcing which does not add sets whose generic is a Bristol sequence.\qed
\end{proposition}
This is done, of course, by approximating the Bristol sequence with set-length initial segments. This is also the way we can remove global choice from the assumptions. Indeed, if $V$ was a model of $\ZFC$, add a generic Bristol sequence, use it to define a Bristol model between $V$ and $V[c]$, where $c$ is a Cohen real, and promptly forget about this generic sequence.

As very clearly the Bristol model is definable from its Bristol sequence in $V[c]$, this means that there may be undefinable Bristol models. And indeed, this can very much be the case.
\begin{theorem}
Suppose that $V$ is a countable transitive model of $\ZFC+\GCH+\square^*_\lambda$ for all singular cardinals $\lambda$, and let $c$ be a Cohen real over $V$. Then there are uncountably many Bristol models intermediate between $V$ and $V[c]$.
\end{theorem}
\begin{proof}
Given two Bristol sequences such that $\cA_0$ and $\cA'_0$ are the two permutable families on $\omega$ and such that $\bigcup\cA_0\cap\bigcup\cA_0'$ is finite, the two Bristol models are distinct. This can be easily arranged, for example taking $\cA_0$ to only have subsets of even integers and $\cA'_0$ only has subsets of odd integers. This can be extended to any other step in the Bristol sequence. We can now easily construct uncountably many distinct Bristol models by simply considering with each subset of $\Ord^V$ how to modify a given, or indeed a generic, Bristol sequence.
\end{proof}

This is in stark contrast to the case of intermediate models of $\ZFC$ of which there are only not just countably many, but there is a set of all the necessary generators in $V[c]$, and the same can be said about symmetric extensions given by the Cohen forcing itself, as Grigorieff's theorem shows. And while there is a proper class of symmetric extensions of $V$, it is still enumerated by the sets in $V[c]$, making it countable in the meta-theory.

Therefore between $V$ and $V[c]$ most models are models of $\ZF$, they are not symmetric extensions of $V$, and in fact they are not definable in $V[c]$. To make matters worse, inside each Bristol model, we can find a different real, and use that real to interpret the Bristol sequence. Just as well, we may also use one of the $\varrho_{\alpha+1}\restriction A^\alpha_\xi$ to interpret the Bristol model construction above a certain stage.

Which truly indicates that the Bristol models are intertwined through this multiverse of models. But to truly appreciate the Bristol model(s), and to understand a bit more its internal structure, we need to have a refined sense of choicelessness.
\section{Kinna--Wagner Principles}\label{sect:kinna-wagner}
The axiom of choice can be simply stated as ``every set can be injected into an ordinal'', or in other words, ``every set is equipotent with a set ordinals''. This, in conjunction with the following theorem, makes a nice way of understanding models of $\ZFC$.

\begin{theorem*}[Balcar--Vop\v{e}nka]
Suppose that $M$ and $N$ are two models of $\ZF$ with the same sets of ordinals. If $M\models\ZFC$, then $M=N$.\qed
\end{theorem*}

Commonly this is stated when $M$ and $N$ are both models of $\ZFC$, but that is in fact unnecessary. The proof relies on the fact that a relation on a set of ordinals can be coded as a set of ordinals. But can we extend the idea of this characterisation? Yes, yes we can.

\begin{definition}
We say that $A$ is an \textit{$\alpha$-set of ordinals}, or simply ``$\alpha$-set'', if there is an ordinal $\eta$ such that $A\subseteq\power^\alpha(\eta)$.
\end{definition}
\begin{definition}
Kinna--Wagner Principle for $\alpha$, denoted by $\KWP_\alpha$, is the statement that every set is equipotent with an $\alpha$-set. We write $\KWP$ to mean $\exists\alpha\, \KWP_\alpha$.
\end{definition}

For $\alpha=0$ this is simply $\AC$. The principle for $\alpha=1$ was defined by Kinna and Wagner, although formulated differently using selection functions, and was studied extensively. One of the immediate results is that $\KWP_1$ implies that every set can be linearly ordered, but as the work of Pincus shows in \cite{Pincus:KW}, $\KWP_1$ is stronger than the existence of linear orderings, and in fact independent of the Boolean Prime Ideal theorem (which also implies every set can be linearly ordered).

These general principles were defined by Monro in \cite{Monro:1973} for $\alpha<\omega$, and later extended by the author in \cite{Karagila:Iterations}. Monro extended the result of Balcar and Vop\v{e}nka, and this result can be further extended as well.

\begin{theorem}[The Generalised Balcar--Vop\v{e}nka--Monro Theorem]\label{thm:gen-bvm}\hfill\newline
Suppose that $M$ and $N$ are models of $\ZF$ with the same $\alpha$-sets. If $M\models\KWP_\alpha$, then $M=N$.
\end{theorem}
\begin{proof}
  Note that for every $\alpha$, we can code relations on $\alpha$-sets by using $\alpha$-sets in a robust and definable way by extending G\"odel's pairing function. So we can simply repeat the proof of Balcar--Vop\v{e}nka. If $x\in M$, we can encode $\tcl(\{x\})$ and its membership relation as an $\alpha$-set, $X$. So $X\in N$, and by decoding the membership relation and applying Mostowski's collapse lemma, $x\in N$. Therefore $M\subseteq N$.

  In the other direction, suppose that $V_\eta^M=V_\eta^N$, then we can encode it as an $\alpha$-set in $M\cap N$. Now given any $x\subseteq V_\eta$ in $N$, i.e.\ $x\in V_{\eta+1}^N$, by looking at the $\alpha$-set encoding $V_\eta$, we can identify the subset corresponding to $x$, in $N$. But as $M$ and $N$ share the same $\alpha$-sets, this implies that this subset is in $M$ as well, and we can therefore find $x\in M$. Now by transfinite induction, $V_\eta^M=V_\eta^N$ for all $\eta$, and equality ensues.
\end{proof}

Monro proved in \cite{Monro:1973} that $\KWP_{n+1}\nto\KWP_n$ for all $n$. This result was extended to show that $\KWP_\omega\nto\KWP_n$ for all $n$ by the author in \cite{Karagila:Iterations}, and was later extended as well by Shani in \cite{Shani:2018} to show that for all $\alpha<\omega_1$, $\KWP_{\alpha+1}\nto\KWP_\alpha$.

\begin{definition}
\textit{Small Violation of Choice} holds if there exists a set $A$ such that for any set $x$ there is an ordinal $\eta$ and a surjection $f\colon\eta\times A\to x$. We write $\SVC(A)$ to specify this set $A$, or $\SVC$ to mean $\exists A\,\SVC(A)$.
\end{definition}
This axiom is due to Blass in \cite{Blass:1979}, and it turns out to play an important role in the study of symmetric extensions. Blass proved in \cite{Blass:1979} that $\SVC$ is equivalent to the statement ``The axiom of choice is forceable [by a set forcing]'', and that $\SVC$ holds in every symmetric extension. Usuba showed in \cite{Usuba:2019} that the latter implication can be reversed. Namely, $V\models\SVC$ if and only if $V$ is a symmetric extension of a model of $\ZFC$.

\begin{proposition}
Suppose that $\SVC(A)$ holds, if $A$ is equipotent with an $\alpha$-set, then $\KWP_{\alpha+1}$ holds.
\end{proposition}
\begin{proof}
We may assume that $A$ is an $\alpha$-set itself, and so if $f\colon\eta\times A\to x$ is a surjection, by coding we may replace $\eta\times A$ by an $\alpha$-set as well, say $A_\alpha$. Now the function mapping $y\in x$ to $y_f=\{a\in A_\alpha\mid f(a)=y\}$ is injective, and each $y_f$ is an $\alpha$-set. Therefore $x$ is equipotent to the $(\alpha+1)$-set $\{y_f\mid y\in x\}$.
\end{proof}
In the other direction, however, the implication fails. It is consistent that $\KWP$ holds, but $\SVC$ fails. For example, in \cite{Monro:1975} a class-symmetric extension is constructed which easily fails $\SVC$, but the proof that $\KWP_1$ holds in the Cohen model can be adapted for it as well.\footnote{In most class-symmetric extensions a similar thing should happen, if each component of the product preserves $\KWP_\alpha$, then the product as a whole will satisfy $\KWP$. We invite the interested reader to prove this.}

Before we return to study the Bristol model, let us study Kinna--Wagner Principles a bit more in depth.

\begin{theorem}\label{thm:kw-generic}
Suppose that $V\models\KWP_\alpha$, if $V[G]$ is a generic extension, then $V[G]\models\KWP_{\alpha^*}$, where $\alpha^*=\sup\{\beta+2\mid\beta<\alpha\}$.\footnote{In other words, $\alpha^*$ is the successor of $\alpha$ when $\alpha$ itself is a successor, or $\alpha$ itself otherwise.}
\end{theorem}
\begin{proof}
  For every $x\in V[G]$ there is a name $\dot x$ in $V$, so $\dot x\restriction G=\{\tup{p,\dot y}\in\dot x\mid p\in G\}$ in $V[G]$, and the interpretation map is a surjection onto $x$, which we can extend to a surjection from $\dot x$ itself. Therefore every set in $V[G]$ is the surjective image of a set in $V$.

  If $\alpha<\alpha^*$ the above completes the proof, as we may assume that $\dot x$ was an $\alpha$-set, and conclude that $x$ is an $(\alpha+1)$-set. For $\alpha=0$ or a limit ordinal, we use \autoref{lemma:inverse-alpha-sets}.
\end{proof}
\begin{lemma}\label{lemma:inverse-alpha-sets}
Let $A$ be an $\alpha$-set for $\alpha=0$ or a limit ordinal. If $f\colon A\to x$ is a surjection, then there exists an $\alpha$-set $B$ which is equipotent with $x$.
\end{lemma}
\begin{proof}
  For $\alpha=0$, $A$ is a set of ordinals, so we may simple choose the least ordinal from the pre-image of each $y\in x$. Suppose that $\alpha$ is a limit ordinal, and define for $\beta<\alpha$, $A_\beta=\{a\in A\mid a\text{ is a }\beta\text{-set}\}$. For every $y\in x$, let $\beta_y$ be the least $\beta$ such that for some $a\in A_\beta$, $f(a)=y$.

  Now define $B_y=\{a\in A_{\beta_y}\mid f(a)=y\}$, this is a $(\beta_y+1)$-set, and $y\neq y'$ implies $B_y\neq B_{y'}$. Therefore $B=\{B_y\mid y\in x\}$ is an $\alpha$-set equipotent with $x$.
\end{proof}
On the other hand, ground models need not satisfy the same $\KWP_\alpha$ as their generic extension: as we have seen in the construction of the Bristol model, $L[c]$ has a proper class of ground models, $L(V_\alpha^M)$, and as we will see in the next section, the various $\KWP_\alpha$ fails as we go up the hierarchy.

The next obvious question is whether or not \autoref{thm:kw-generic} can be improved. Unfortunately, it cannot. The Cohen model famously satisfies $\KWP_1$,\footnote{An implicit proof can be found as Lemma~5.25 in \cite{Jech:AC1973}.} but as Monro demonstrated in \cite{Monro:1983}, there is a generic extension of the Cohen model in which there exists an amorphous set, which cannot be linearly ordered, in particular, $\KWP_1$ must fail in that generic extension. Nevertheless, by the theorem above, $\KWP_2$ must hold. This leads us to the following conjecture.\footnote{These are not directly related to the Bristol model, so we include them here instead of \autoref{sect:questions}}
\begin{conjecture}[The $\alpha^*$ Conjecture]
Suppose that $\KWP_{\alpha^*}$ holds in every generic extension of $V$, then $\KWP_\alpha$ must hold in a ground of $V$.\footnote{A previous version of this manuscript suggested that the ground can be taken as $V$ itself. However Monro's generic extension of the Cohen model satisfies $\KWP_2+\lnot\KWP_1$, and any generic extension of that model is also a generic extension of the Cohen model, so will also satisfy $\KWP_2$.}
\end{conjecture}

It is also easy to see that if $M\models\lnot\KWP$, then also any generic extension of $M$ must satisfy this. Which points out to a particularly poignant feature of a generic multiverse, and indeed a \textit{symmetric multiverse},\footnote{Allowing symmetric extensions and grounds. See \cite{Usuba:2019} for more information.} $\KWP$ hold or fails uniformly throughout the entire multiverse.

\begin{conjecture}[The Kinna--Wagner Conjecture]
Suppose that $V\models\KWP$ and $G$ is a $V$-generic filter. If $M$ is an intermediate model between $V$ and $V[G]$ and $M\models\KWP$, then $M=V(x)$ for some set $x$.
\end{conjecture}
The Bristol model was an exercise in finding an intermediate model which is not constructible from a set. And we conjecture that having any such model, intermediate to a generic extension, will fail $\KWP$, and vice versa: any intermediate model of $\KWP$ is constructible from a set over the ground model.
\begin{remark}
  The Kinna--Wagner Conjecture was proved by the author and Jonathan Schilhan in \cite{KaragilaSchilhan:2025}, along with a more thorough study of Kinna--Wagner Principles in the generic multiverse, including improvements to several of the theorems presented in this section.
\end{remark}

It may also be the case that $\KWP$ implies ground model definability, which is a notoriously difficult problem in $\ZF$. Usuba proved \cite[Corollary~6]{Usuba:LSP} that under a certain condition ground models are definable in $\ZF$, but currently the only known models to satisfy these conditions are models satisfying $\SVC$.\footnote{The existence of a proper class of L\"owenheim--Skolem cardinals, which follows from $\SVC$.} Nevertheless, as $\alpha$-sets can be used to characterise a model of $\ZF$ in the presence of $\KWP_\alpha$, it stands to reason that it may play a role in ground model definability as well.

We can also define $\SVC_\alpha$ to mean that we replace the ordinal, $\eta$, by an $\alpha$-set. And in that case we can easily see that $\SVC_\alpha$ is equivalent to ``$\KWP_{\alpha^*}$ is forceable''. And one is now left wondering if $\SVC_\alpha$ is equivalent to being a symmetric extension of a model satisfying $\KWP_\alpha$.

One can also take a different approach and define $\SVC_M$, for a class $M$, where we may replace the ordinal $\eta$ by a set from $M$, so $\SVC=\SVC_\Ord$ and $\SVC_\alpha$ is a shorthand for $\SVC_{\power^\alpha(\Ord)}$. For this to be truly useful, we need to modify $\KWP_\alpha$ so that $0$-sets are subsets of $M$. These ideas may play a role in the ultimate answers regarding ground model definability in $\ZF$, and we hope this discussion will help to inspire some of the readers to think about that.

\begin{question}
  Is ground model definability equivalent to $\KWP$?
\end{question}
Note that this question is meaningful since as we observed, $\KWP$ is absolute through the generic, and indeed the symmetric, multiverse.
\section{Choice principles in Bristol}\label{sect:choice}
We want to investigate the failure of the axiom of choice in the Bristol model, $M$, and provide alternative proofs to the key property of the Bristol model, namely $\forall x(V\neq L(x))$.

\begin{proposition}[\cite{Karagila:Bristol}, Theorem~5.5]
Let $M$ denote the Bristol model, and $M_\alpha$ be the $\alpha$th model in the construction. Suppose that $A\in M$ is an $\alpha$-set, then $A\in M_{\alpha+1}$.
\end{proposition}
We will not prove this proposition here, but the idea extends the homogeneity argument used in \autoref{prop:reals-support}. Indeed, $L(V_{\alpha+1}^M)$ contains all the $\alpha$-sets of $M$.
\begin{corollary}
  Suppose that $W$ is a model of $\ZFC+\GCH+\square^*_\lambda$ for every singular cardinal $\lambda$, and let $M$ be a Bristol extension of $V$. Then $W$ is definable in $M$.
\end{corollary}
\begin{proof}
Note that $N=L(\power(\Ord))^M$ contains $W$, since it contains $\power(\Ord)^M$ and $W\models\ZFC$. It is therefore equal to $W(V_{\omega+1}^M)$, since by \autoref{prop:0-sets} every set of ordinals is either in $W$ or in $W[r]$ for some real. That is to say, $N$ is a symmetric extension of $W$. By Usuba's result \cite[Corollary~5.6]{Usuba:2019}, $W$ is definable in its symmetric extensions, so $W$ is definable in $N$, and therefore in $M$.
\end{proof}
\begin{corollary}
If $\beta>\alpha$, then $M_\beta\models\lnot\KWP_\alpha$. In particular, $M\models\lnot\KWP$.
\end{corollary}
\begin{proof}
$M,M_\beta$ and $M_{\alpha+1}$ have the same $\alpha$-sets. If one of them would satisfy $\KWP_\alpha$, by \nameref{thm:gen-bvm}, $M=M_\beta=M_{\alpha+1}$.
\end{proof}
\begin{corollary}
$M\models\lnot\SVC$, and consequently $M\models\forall x(V\neq L(x))$, as well as ``the axiom of choice is not forceable (by a set forcing)''.\qed
\end{corollary}

It follows from this that at least for a proper class $A\subseteq\Ord$, if $\alpha<\beta$ are both in $A$, then $\KWP_\beta\nto\KWP_\alpha$. But we want to understand the gradation, or rather the degradation, of $\KWP$ through the construction.
\begin{proposition}
$M_1\models\lnot\KWP_1\land\lnot\BPI$
\end{proposition}
\begin{proof}
  We proved that $R$, the set of $R_\xi=V_{\omega+1}^{L[c\cap A_\xi^0]}$ for $\xi<\omega_1$, is $\aleph_1$-amorphous in $M_1$ in \autoref{prop:R-amorphous}, and the same argument shows that it cannot be linearly ordered. In particular, it cannot be equipotent to any $1$-set and it also witnesses that $\BPI$ fails.

  Briefly, the argument starts by taking a name $\dot\prec\in\IS_1$ and some $p\forces^\IS``\tup{\dot R,\dot\prec}^\bullet$ is a linearly ordered set''. Let $\cB$ be a disjoint approximation such that $\fix(\cB)\subseteq\sym(\dot R)$ and $\dom p\subseteq\bigcup\cB$. We pick $\alpha,\beta\notin\dom\cB$ and distinct, and we let $q\leq p$ decide, without loss of generality, $\dot R_\alpha\mathrel{\dot\prec}\dot R_\beta$. But now we can find $\pi\in\fix(\cB)$ such that $\iota(\pi)=(\alpha\ \beta)$ and $\pi q=q$, which is impossible if $\dot\prec$ was a name for a linear order.
\end{proof}
This is a remarkable point, as the elements of $R$ themselves can be well-ordered separately. So you may think we can replace them by $0$-sets, but we cannot do that uniformly, and this forces us to treat them as $1$-sets instead, and $R$ itself as a $2$-set.
\begin{proposition}\label{prop:m2-bpi}
$M_2\models\lnot\BPI$.
\end{proposition}
\begin{proof}
  We show that $R$ still cannot be linearly ordered when passing to $M_2$. Of course the forcing that led us there, $\QQ_1$, linearly orders (and in fact well-orders) $R$. But it is easy to see that this well-order is promptly discarded, and instead we only remember bits and pieces of it in the form of $\varrho_1\restriction A^1_\xi$. To show that there is no linear ordering in $M_2$ we need to use the full power of the symmetric iteration. We are going to start with a rather naive attempt, which may not work, but we can identify the problem and circumvent it.

  Suppose that $\dot\prec\in\IS_2$ is such that $\tup{p,\dot q}\forces^\IS_2``\dot\prec$ linearly orders $\dot R$''. Let $\cB_0$ and $\cB_1$ be disjoint approximations such that $\tup{\fix(\cB_0),\fix(\cB_1)}$ is a support of $\dot\prec$. Let $\alpha,\beta<\omega_1$ be such that $\alpha,\beta\notin\dom\cB_0\cup\bigcup\cB_1$ and moreover $\dot R_\alpha,\dot R_\beta$ are not mentioned in $\dot q$.

  Let $\tup{p',\dot q'}$ be a condition extending $\tup{p,\dot q}$ such that $\tup{p',\dot q'}\forces^\IS_2\dot R_\alpha\mathrel{\dot\prec}\dot R_\beta$. We would like to apply upwards homogeneity and consider $\pi\in\fix(\cB_0)$ which implements $(\alpha\ \beta)$ while also not moving $p'$. But we have to contend with the fact that $\pi\dot q'$ may have moved. Luckily, we know exactly where it moved to: $\pi$ simply permutes the range of $\dot q$, so if $\dot R_\alpha$ and $\dot R_\beta$ appear in $\dot q'$, which is the likely case, then we simply need to switch them back using some $\sigma\in\fix(\cB_1)$, that is an automorphism of $\dot\QQ_1$, and that will be enough, since automorphisms of $\dot\QQ_1$ do not change $\dot R_\alpha$ and $\dot R_\beta$.

  Alas, we have a problem. For $\sigma$ to exist, we need to make sure that $\dot R_\alpha$ and $\dot R_\beta$ appear in $\dot q'$ in coordinates which are not in $\bigcup\cB_1$, otherwise we cannot move these coordinates at all. So this naive approach cannot work.

  Luckily, we assumed that $\dot R_\alpha$ and $\dot R_\beta$ are not mentioned in our original $\dot q$. So to find $\dot q'$, first add both of these in coordinates that are not in $\bigcup\cB_1$, and if this was not enough to decide how $\dot\prec$ will order them we can extend further to find $\dot q'$. In other words, we may assume without loss of generality that $\dot q'$ mentions $\dot R_\alpha$ and $\dot R_\beta$ in coordinates which are eligible to be switched from within $\fix(\cB_1)$.

  Therefore, if $\tup{p',\dot q'}\forces^\IS_2\dot R_\alpha\mathrel{\dot\prec}\dot R_\beta$, then also $\tup{p',\dot q'}\forces^\IS_2\dot R_\beta\mathrel{\dot\prec}\dot R_\alpha$. Therefore $\tup{p',\dot q'}$ cannot force $\dot\prec$ to be a linear ordering, which is a contradiction, since it extends a condition which did force just that.
\end{proof}
One may think that this is enough to prove that there are sets which cannot be linearly ordered in the Bristol model, as we just exhibited that the second symmetric extension will not linearly order $R$ either. Alas, we already concluded that $R$ is a $2$-set, so a linear ordering of $R$ will also be a $2$-set. But we know that $2$-sets are only determined in $M_3$. But luckily, we are not very far behind completing this part of the journey.
\begin{theorem}\label{thm:no-bpi}
$M\models\lnot\BPI$, and in fact there is a set in $M$ which cannot be linearly ordered.
\end{theorem}
\begin{proof}
  Suppose that $\tup{p,\dot q,\dot r}\in\PP_3$ is a condition that for some $\dot\prec\in\IS_3$ forces that $\dot\prec$ is a linear order of $\dot R$. We can actually run the proof of \autoref{prop:m2-bpi} again. First of all, $\pi$ will not affect the condition extending $\dot r$ at all, but more importantly, $\sigma$ can be chosen as a permutation moving only two points which means it implements the identity. So it also will not modify the condition extending $\dot r$.

  And so as long as we were careful to choose the extension $\dot q'$ in a way that allows $\sigma$ to be taken from $\fix(\cB_1)$, the argument is not affected. Therefore we showed that $M_3\models``R$ cannot be linearly ordered'', and therefore $M$ does as well.
\end{proof}
\section{Open questions related to the Bristol model}\label{sect:questions}
There is still so much to learn about the Bristol model, both in the specific context of $L$, as well as many natural questions that come up from generalisations and details in the proof. We cannot possibly include all of these, but we will give four families of questions which are interconnected, but also seem to have independent interest.
\subsection{The Bristol models in the multiverse}
\begin{question}
  Is there a condition characterising the equivalence classes on Bristol sequences (definable or otherwise) based on the Bristol model they generate using a fixed Cohen real?
\end{question}
\begin{question}
  Is the theory of any two Bristol models the same? Does the theory depend on the sequence or its properties?
\end{question}
\begin{question}
  Are there any non-trivial grounds of any Bristol model?
\end{question}
\begin{question}
  Is there a Bristol model which is definable in its generic extensions, or maybe is there one that is not definable in some of its generic extensions?
\end{question}
\begin{question}
  Is there a generic extension of a Bristol model which itself is a Bristol model?
\end{question}
\begin{question}
  Is there a maximal Bristol model, namely, is there a Bristol model $M\subseteq L[c]$ such that for any $x\in L[c]\setminus M$, $M(x)=L[c]$?
\end{question}
\begin{question}
  Is it true in general that for $x\in L[c]$, either $L(x)=L[c]$ or there is a Bristol model $M$ such that $x\in M$?
\end{question}
\subsection{Large cardinals in the Bristol model}
\begin{question}
  We saw that measurable cardinals remain measurable. Do they remain critical cardinals in the sense of \cite{HayutKaragila:Critical}? What about weakly critical cardinals?
\end{question}
\begin{question}
  Principles like $\square^*_\lambda$ are considered to witness failure of compactness. Suppose that $\lambda$ is singular and no permutable scale exists on $\prod\SC(\lambda)$. Can this compactness be harnessed to restart the Bristol model construction? (Note that a positive answer would indicate that Woodin's Axiom of Choice Conjecture is possibly false, which may imply also the eventual failure of the HOD Conjecture. As such the answer to this question is most likely negative, and a positive answer would be extremely hard to prove.)
\end{question}
\begin{question}
  Suppose that elementary embeddings can be lifted and Woodin cardinals are preserved. Starting from strong enough hypotheses, can we construct Bristol model-like objects that satisfy $\AD$?
\end{question}
\subsection{Other type of Bristol models}
\begin{question}
  Can we start the construction of the Bristol model with a different type of real? Clearly not every real is useful, minimal reals do not have intermediate models, for example. But what about reals that admit sufficiently many automorphisms and intermediate models such as random reals? What about ``Cohen + condition'' type of reals (Hechler, Mathias, etc.)? Will this also impact the type of forcing we need to do in the following steps (namely, will that force us to use something which is not isomorphic to $\Add(\omega_\alpha,1)^L$ in successor steps)?
\end{question}
\begin{question}
  Can we start the construction with a Prikry-like forcing instead of a Cohen forcing?
\end{question}
\begin{question}
  While there is no good definition for iteration of symmetric extensions with countable support, it is imaginable that for productive iterations such as the one used in the Bristol model this is doable by hand. What would this be? Can we have an $\omega_1$-Bristol model starting with an $L$-generic sequence for $\Add(\omega_1,1)$ for example?
\end{question}
\begin{question}
  Can we find similar constructions over arbitrary models of $\ZF$, or at least some reasonable combinatorial property which involves the $V_\alpha$ hierarchy, rather than the ordinals?
\end{question}
\subsection{Weak choice principles}
\begin{question}
  Does $\DC$ hold in the Bristol model?\footnote{In an upcoming work with Jonathan Schilhan we answer this question to the positive. In particular, countable choice holds in the Bristol model. We do not know about $\AC_\WO$.}
\end{question}
\begin{question}
  Are there any free ultrafilters on $\omega$ in the Bristol model?
\end{question}
\begin{question}
  Does $M_\alpha\models\KWP_{\alpha*}$?
\end{question}
\begin{question}
  Are there any choice principles that can be forced over the Bristol model?
\end{question}
\begin{question}
  Is countable choice true in the Bristol model? If so, is $\AC_\WO$, the axiom of choice for families that can be well-ordered, true? (Note that this will provide a positive answer about $\DC$, as well as the lifting of elementary embeddings for measurable cardinals.)
\end{question}
\begin{question}
  Say that $A$ is \textit{$x$-amorphous} if it cannot be well-ordered, and it cannot be written as the union of two sets that are not well-orderable. That is, for some $\kappa$, $A$ is $\kappa$-amorphous. Are there any $x$-amorphous sets in the Bristol model?
\end{question}
\bibliographystyle{amsplain}
\providecommand{\bysame}{\leavevmode\hbox to3em{\hrulefill}\thinspace}
\providecommand{\MR}{\relax\ifhmode\unskip\space\fi MR }
\providecommand{\MRhref}[2]{%
  \href{http://www.ams.org/mathscinet-getitem?mr=#1}{#2}
}
\providecommand{\href}[2]{#2}

\end{document}